\documentclass[11pt]{article}
\usepackage[centertags]{amsmath}
\hsize \textwidth
\advance \hsize by -\marginparwidth
\evensidemargin
\oddsidemargin
\usepackage{amsmath,amsthm}    
\usepackage{amssymb}

\advance\hoffset by 5mm

\usepackage{amssymb}

\newtheorem{theorem}{Theorem}[section]
\newtheorem{lemma}[theorem]{Lemma}

\newtheorem{corollary}[theorem]{Corollary}

\newtheorem{prop}[theorem]{Proposition}
\newtheorem{cor}[theorem]{Corollary}

\newcommand{\wT}{\widetilde{T}}
\newcommand{\wG}{\widetilde{G}}
\newcommand{\tl}{\tilde{\lambda}}
\newcommand{\bT}{\bar{T}}

\numberwithin{equation}{section}
\newcommand{\R}{{\mathbb{R}}}
\newcommand{\Rm}{{\mathbb{R}}}
\newcommand{\bD}{{\mathbb{D}}}
\newcommand{\Pm}{{\mathbb{P}}}
\newcommand{\E}{{\mathbb{E}}}
\newcommand{\N}{{\mathbb{N}}}
\newcommand{\bE}{{\mathbf{E}}}

\newcommand{\1}{{\mathbf{1}}}
\newcommand{\eps}{\varepsilon}
\newcommand{\farc}{\frac}
\newcommand{\wK}{\widetilde K}
\newcommand{\hCI}{\widehat{\mathcal{C}}_I}
\newcommand{\hBI}{\widehat{\mathcal{B}}_I}

\begin{document}

\title{The speed of a random front for stochastic reaction-diffusion
equations with strong noise}
\author{Carl Mueller\thanks{Department of Mathematics, University of Rochester,
Rochester, NY, USA  14627;  carl.e.mueller@rochester.edu}
\and Leonid Mytnik\thanks{Faculty of Industrial Engineering and Management,
Technion, Technion City, Haifa 3200003, Israel; leonid@ie.technion.ac.il} 
\and Lenya Ryzhik\thanks{Department of Mathematics,
Stanford University,
Stanford CA 94305, 
USA; ryzhik@math.stanford.edu}}

\maketitle

\begin{abstract}
We study the asymptotic speed of a random front for solutions $u_t(x)$ to
stochastic reaction-diffusion equations of the form 
\[
\partial_tu=\farc{1}{2}\partial_x^2u+f(u)+\sigma\sqrt{u(1-u)}\dot{W}(t,x),~t\ge 0,~x\in\Rm,
\]
arising in population genetics.  Here, $f$ is a  continuous function with 
$f(0)=f(1)=0$, and such that~$|f(u)|\le K|u(1-u)|^\gamma$
with~$\gamma\ge 1/2$,  and $\dot{W}(t,x)$ is a space-time Gaussian white 
noise.  We assume that the initial condition $u_0(x)$ satisfies 
$0\le u_0(x)\le 1$ for all $x\in\Rm$, $u_0(x)=1$ for~$x<L_0$ and 
$ u_0(x)=0$ for~$x>R_0$.   We show that when $\sigma>0$, for each $t>0$ 
there exist~$R(u_t)<+\infty$ and~$L(u_t)<-\infty$ such that $u_t(x)=0$ for 
$x>R(u_t)$ and $u_t(x)=1$ for~$x<L(u_t)$ even if $f$ is not Lipschitz. We 
also show that for all $\sigma>0$ 
there exists a finite deterministic speed~$V(\sigma)\in\Rm$ so 
that~$R(u_t)/t\to V(\sigma)$ as $t\to+\infty$, almost surely. This is in 
dramatic contrast
with the deterministic case $\sigma=0$ for nonlinearities of the type
$f(u)=u^m(1-u)$ with $0<m<1$ when solutions converge to $1$
uniformly on $\Rm$ as $t\to+\infty$.  Finally, we 
prove that when $\gamma>1/2$ there exists $c_f\in\Rm$, so that~$\sigma^2V(\sigma)\to c_f$ 
as~$\sigma\to+\infty$ and give a characterization of $c_f$.  
The last result complements a lower bound obtained by Conlon and Doering 
\cite{cd05} for the special case of $f(u)=u(1-u)$ where a duality argument
is available.
\end{abstract}

\section{Introduction}
\label{sec:introduction}

Reaction-diffusion equations of the form
\begin{equation}\label{feb1120}
\partial_tu=\frac{1}{2}\partial_x^2u+f(u),
\end{equation}
with $f(0)=f(1)=0$, are often used to model biological invasions and other
spreading phenomena, with one steady state, say, $u\equiv 1$ invading another,
$u\equiv 0$, or vice versa. Under very mild assumptions on $f(u)$, such as, for
instance, that $f(u)$ is Lipschitz on $[0,1]$ and either $f(u)>0$ for $u\in(0,1)$,
or there exists $\theta\in(0,1)$ so that $f(u)\le 0$ for $u\in(0,\theta)$ and $f(u)>0$
for $u\in(\theta,1)$, 
such equations admit traveling wave solutions of the form $u_t(x)=U(x-ct)$ such that
\begin{equation}\label{feb1122}
-cU'=\farc{1}{2}U''+f(U),~U(-\infty)=1,~U(+\infty)=0.
\end{equation}
Note that, in the probabilistic spirit of the present paper, the subscript $t$
denotes the time dependence of the function $u_t(x)$ rather than a time derivative,
common to the PDE literature. 
It is easy to see that
\begin{equation}\label{feb1124}
c\int_{\R}|U'(x)|^2dx=\int_0^1f(z)dz,
\end{equation}
thus $c$ has the same sign as
\begin{equation}\label{feb1126}
I[f]:=\int_0^1f(u)du,
\end{equation}
so that if $I[f]>0$ then the steady state $u\equiv 1$ is more stable, and invades
the "less stable" steady state~$u\equiv 0$, and if $I[f]<0$ then the opposite happens,
while if $I[f]=0$ then (\ref{feb1120}) has a time-independent solution. It is also
well-known that traveling wave solutions to (\ref{feb1120}) determine the spreading
speed for the solutions of the Cauchy problem. More precisely, let $u_t(x)$ be the solution
to~(\ref{feb1120}) with an initial condition $u_0(x)$ such that $0\le u_0(x)\le 1$
for all $x\in\R$, and there exist~$L_0\le R_0$ so that $u_0(x)=1$ for $x<L_0$  
and $u_0(x)=0$ for $x> R_0$. There exists a function $m(t)$ such that
\begin{equation}\label{feb1302}
|m(t)-c_*t|=o(t)\hbox{ as $t\to+\infty$},
\end{equation}
so that
\begin{equation}\label{feb1304}
|u_t(x+m(t))-U_{c_*}(x)|=o(1)\hbox{ as $t\to+\infty$}.
\end{equation}
Here, depending on the nature of the nonlinearity $f(u)$, the spreading speed
$c_*$ may be either the speed of the unique traveling wave, or the minimal
speed of a traveling wave if traveling waves are not unique. The latter happens
for the class of the Fisher-KPP nonlinearities, such that $f$ is Lipschitz,
$f(0)=f(1)=0$, $f(u)>0$ for all $u\in(0,1)$, and $f(u)\le f'(0)u$ for all $u\in[0,1]$. 
In that case, we have 
\begin{equation}\label{feb1130}
c_*=\sqrt{2f'(0)}.
\end{equation}
Much more precise results than (\ref{feb1302})-(\ref{feb1304})  
on the convergence of the solutions to the Cauchy problem
to traveling waves are available,
and we refer to the classical papers~\cite{Ar-W,Bramson1,Bramson2} for the basic results,
and to~\cite{NRR,Roberts} and references therein for more recent developments.
We also point out the relation
\begin{equation}\label{feb1308}
c_*=\lim_{t\to+\infty}\int_\R f(u_t(x))dx=\int_\R f(U_{c_*}(x))dx,
\end{equation}
that can be obtained simply by integrating (\ref{feb1120}) and (\ref{feb1122}) in space.

Note that if $f'(0)$ blows up, then the speed of propagation may also tend to infinity, as can be seen
from (\ref{feb1130}). For H\"older nonlinearities such that $f(u)\sim u^p$ with $p\in(0,1)$, it was shown in~\cite{AE}
that solutions become instantaneously strictly positive everywhere:
$u(t,x)\ge ct^{1/(1-p)}$ for~$t\ll 1$. In particular, if we approximate such nonlinearity by a sequence of Lipschitz
nonlinearities $f_n$, then the corresponding spreading speeds $c_*^{(n)}$ blow up as $n\to+\infty$. 

\subsubsection*{Reaction-diffusion equations with noise}

The physical and biological systems modeled by reaction-diffusion equations
are often subject to 
noise. 
In this paper, we study solutions $u_t(x)$,  
to the stochastic reaction-diffusion equations of the form
\begin{equation}
\label{eq:RFKPP-general}
\partial_tu=\farc12\partial_x^2u+f(u)+\sigma\sqrt{u(1-u)}\dot{W}(t,x)
\end{equation}
where $\dot{W}(t,x)$ is a space-time Gaussian white noise,
and $\sigma>0$ measures its strength. Our interest is in the 
effect of the noise term on the spreading speed.   Since traveling waves will 
no longer maintain a fixed shape due to the noise, we will refer instead to
the speed of the random front, which is defined below.

Let us give an motivation for the noise term in \eqref{eq:RFKPP-general}
similar to that given by Fisher in his pioneering work \cite{fis37}.  See 
also \cite{Shiga-rev}.  Imagine that two populations, type A and type B, move in 
a Brownian way along ${\R}$,  
and let $u_t(x)$ is the proportion of the population of type A at time $t$ at position $x$. 
When an individual of type A meets an individual of type
B, it can be converted into
type B, and vice versa, and the outcome is partially random.
The function $f(u)$ in (\ref{eq:RFKPP-general})
describes the deterministic evolution of the population of type~A,
due to these interactions,
and it is natural to assume that $f(0)=f(1)=0$ since there are no interactions when
one type is absent.
The random term in \eqref{eq:RFKPP-general} accounts for the
stochastic aspect of the interactions.  We assume that 
for each such meeting we have a mean-zero random variable affecting the outcome, and  
these random variables are i.i.d.  By the central limit theorem, the 
sum of such variables would be approximately Gaussian.  The 
independence of the variables means that the random input should be  
independent for different values of $t$ and $x$, giving rise to the
space-time noise $\dot{W}(t,x)$.  The rate of such meetings at a 
given site $x$ and time $t$ would be proportional to~$u_t(x)(1-u_t(x))$, which  
is  the variance of the noise at~$(t,x)$.  Thus we should multiply the white 
noise $\dot{W}(t,x)$ by the standard deviation $\sqrt{u_t(1-u_t)}$.  This leads 
to the noise term in \eqref{eq:RFKPP-general}.  

As we have mentioned, we are interested in the long time speed of a 
random front for the solutions to~(\ref{eq:RFKPP-general}). To this end, we 
define the left and the right edge 
of the solution as follows.  Given a function~$h(x)$ such that
$0\le h(x)\le 1$ for all $x\in\R$, with $h(x)\to 1$ as $x\to -\infty$ and
$h(x)\to 0$ as~$x\to+\infty$, we set
\begin{align}
\label{eq:def-front}
L(h)&=\inf\left\{x\in{\R}: h(x)<1\right\}  \\
R(h)&=\sup\left\{x\in{\R}: h(x)>0\right\}.  \nonumber
\end{align}
In the absence of the noise, when $\sigma=0$, and for Lipschitz nonlinearities 
$f(u)$, we have $L(u_t)=-\infty$ and $R(u_t)=+\infty$ for all $t>0$. This, however,
is not necessarily the case in the presence of the noise. 
In order to make this claim precise, we assume that  
\begin{equation}\label{feb1202}
\text{$f$ is continuous 
on $[0,1]$ and there exists $K_f>0$ such that 
$f(u)\leq K_f \sqrt{|u(1-u)|}.  $}
\end{equation}
As for the initial condition $u_0(x)$, we will assume that 
\begin{equation}\label{feb1206}
0\le u_0(x)\le 1\hbox{ for all $x\in\Rm$, and both $L(u_0)$ and $R(u_0)$ are finite.}
\end{equation}
We will denote by $\mathcal{C}_I$ the set of continuous functions satisfying (\ref{feb1206}).
 In addition $\hBI$  will denote the space of functions on $\R$ taking values in $[0,1]$ and 
 $\hCI$ will denote the space of continuous  functions on $\R$ taking values in $[0,1]$.

We say that $u_t$ has a speed $V(\sigma)$ if the following limit exists:
\[
V(\sigma)=\lim_{t\to\infty}\frac{R(u_t)}{t}. 
\]
We prove the following theorem in Section~\ref{sec:defs}. 
\begin{theorem}\label{thm-feb2existence}
Let $f(u)$ satisfy (\ref{feb1202}) and $u_0(x)$ be as in (\ref{feb1206}),
then (\ref{eq:RFKPP-general}) with an initial
condition~$u_0(x)$ has a solution $u_t(x)$  taking values in  $\hCI$ for $t>0$. 
The solution is unique in law. Moreover,~$L(u_t)$ and $R(u_t)$ are almost 
surely finite for all $t\ge 0$ and the
 solution has a speed $V(\sigma)\in\R$.  
\end{theorem}
We see that the noise has a very strong slowdown effect: $V(\sigma)$ is finite for all $\sigma>0$ 
even if $f(u)$ is H\"older with an exponent $m\ge 1/2$, and not Lipschitz, 
such as, for instance $f(u)=u^m(1-u)$, for which, as we have mentioned, the speed of the front 
is infinite when $\sigma=0$.  

Most of the papers dealing with \eqref{eq:RFKPP-general}, such as 
Mueller and Sowers~\cite{ms95} have treated the Fisher-KPP nonlinearity
$f(u)=u(1-u)$, and small noise, where 
$\sigma$ is close to 0.  Mueller, Mytnik, and Quastel \cite{mmq11} studied 
the behavior of $V(\sigma)$ as $\sigma\downarrow0$ and verified 
some conjectures of Brunet and Derrida \cite{bd97} and 
\cite{bd00}.  Less attention has been devoted to $V(\sigma)$ for large or intermediate
values of $\sigma$, but Conlon and Doering \cite{cd05} proved 
that for $f(u)=u(1-u)$ there exists an asymptotic velocity $V(\sigma)>0$ for solutions $u$ to 
\eqref{eq:RFKPP-general}  for all $\sigma>0$, 
and that 
\begin{equation}\label{feb12condoe}
\liminf_{\sigma\to\infty}\sigma^2V(\sigma)\geq 1.
\end{equation}
Note that (\ref{feb12condoe}) differs from (1.7) in~\cite{cd05}
because the diffusivity in that paper is taken to be $1$ rather than $1/2$ as 
chosen here.  
To formulate our main result, we 
note that a  rescaling of \eqref{eq:RFKPP-general}, discussed in 
Section~\ref{sec:defs} allows us to move the noise coefficient into the 
nonlinearity, and obtain the rescaled equation
\begin{equation}
\label{eq:spde-v_f}
\partial_tv=\frac{1}{2}\partial_x^2v+\sigma^{-4}f(v)+\sqrt{v(1-v)}\dot{W}(t,x).
\end{equation}
Here $v$ is a rescaling of $u$ which we specify later.  
Later we will use  
the results of Tribe \cite{tri95}, and Mueller and Tribe \cite{mt97} for 
(\ref{eq:spde-v_f}) with $f=0$, a version of a continuous voter model, or a stepping stone model in population genetics:
\begin{equation}
\label{eq:spde-w}
\partial_tw=\frac{1}{2}\partial_x^2w+\sqrt{w(1-w)}\dot{W}(t,x).
\end{equation}
By Theorem 1 of \cite{mt97}, we know that 
$w_t(x-R(w_t))$ converges weakly to a stationary distribution as $t\rightarrow\infty$. 
We denote the 
expectation with respect to the stationary distribution of $w$ by $\E_{w,st}$, where   
"st" is an abbreviation for "stationary". For the next theorem we need an assumption on $f$ which is  slightly stronger than 
\eqref{feb1202}: we assume
\begin{equation}\label{feb1202a}
\text{$f$ is continuous 
on $[0,1]$ and there exists $\wK_f>0$ s.t. 
$f(u)\leq \wK_f |u(1-u)|^\gamma  $ for some $\gamma\in (1/2, 1]$.}
\end{equation}

\begin{theorem}
Suppose that $u_0$ satisfies (\ref{feb1206}) and $f$ satisfies  
(\ref{feb1202a}). Then we have, almost surely,
\label{th:2}
\begin{equation}\label{feb1310}
\lim_{\sigma\to\infty}\sigma^2V(\sigma)=c_f,
\end{equation}
where 
\begin{equation}
\label{eq:def-c-sub-f}
c_{f} \equiv  \E_{w,st}\left[\int_{{\R}}f(w(x))\,dx\right]
\end{equation}
and
\begin{equation}
\label{eq:def-c-sub-f2}
|c_f|<\infty.  
\end{equation}
\end{theorem}
Note that Lemma~2.1 of \cite{tri95} shows that
\begin{equation}
\label{eq:1_5}
\lim_{t\to\infty}\E_{w}\left[\int_{{\R}}w_t(x)(1-w_t(x))\,dx\right]=1.
\end{equation}
This immediately implies  that $|c_f|<\infty$ for $f$ satisfying (\ref{feb1202a}) with $\gamma=1$. In particular, as a consequence of Theorem~\ref{th:2},
we get that for the Fisher-KPP nonlinearity $f(u)=u(1-u)$, we have
\[
\lim_{\sigma\to\infty}\sigma^2V(\sigma)=1,
\]
giving a matching upper bound to the lower bound (\ref{feb12condoe})
of Conlon and Doering in~\cite{cd05}, after adjusting for the different diffusivities adopted in the present paper and in~\cite{cd05}.
For the general~$f$ satisfying (\ref{feb1202a}), we  show that \eqref{eq:def-c-sub-f2} holds  in Lemma~\ref{lem:feb19_2}. 

We also see the slowdown due to strong noise in Theorem~\ref{th:2} even 
for Lipschitz nonlinearities. 
The large noise asymptotics in (\ref{feb1310})
corresponds to the speed of the front for solutions of (\ref{eq:spde-v_f}) 
that is $V^{(v)}(\sigma)\sim c_f/\sigma^4$. However, solutions
of the corresponding equation without the noise
\begin{equation}
\label{eq:spde-v_fbis}
\partial_tv=\frac{1}{2}\partial_x^2v+\sigma^{-4}f(v) 
\end{equation}
spread with the speed $\bar V(\sigma)=c_*/\sigma^2$, where $c_*$ is the speed of the traveling wave for (\ref{eq:spde-v_fbis}) with~$\sigma=1$,
so that $V^{(v)}(\sigma)\ll \bar V(\sigma)$ for $\sigma\gg 1$,
and the noise slows down the propagation.

Let us also point out that expression (\ref{feb1310})-(\ref{eq:def-c-sub-f}) 
for the front speed $V(\sigma)$ is a direct analog of~(\ref{feb1308}) except now
the role of the traveling wave is played by the invariant measure
of $w_t(x)$. One may conjecture that
instead of the convergence to a traveling wave 
in shape, as in~(\ref{feb1304}) that happens in the deterministic case, here, 
in the limit $\sigma\to+\infty$, the law of
$u_t(x)$ after rescaling converges, as~$t\to+\infty$, in the frame moving with the speed $V(\sigma)$,
to the invariant distribution of~$w_t(x)$. 

Another interesting observation  is that the noise, despite its symmetry with respect to $u=0$ and~$u=1$ 
can change the direction of the invasion. One may construct a nonlinearity $f$ such that~$I(f)$ given 
by~(\ref{feb1126})
has a different sign than $c_f$, meaning that that the speed of propagation~for~$\sigma=0$, in the absence of the noise,
may have a different sign than $V(\sigma)$ 
for large $\sigma\gg 1$, changing the direction
of the invasion, because of the noise.

The paper is organized as follows. The proof of Theorem~\ref{thm-feb2existence} is in Section~\ref{sec:defs}.
Section~\ref{sec:voter} contains some auxiliary results on solutions to (\ref{eq:spde-w}). They are used later in the proof of 
Theorem~\ref{th:2}, presented in Sections~\ref{sec:upper} for the upper bound, and in Section~\ref{sec:lower}
for the matching lower bound on the speed $V(\sigma)$ for $\sigma\gg 1$.

\textbf{Acknowledgement.} The work of LM and LR was supported by a US-Israel BSF grant.
LR was supported by NSF grant DMS-1613603 and ONR grant N00014-17-1-2145,
and CM was supported by a Simons Grant.  

\section{The proof of Theorem~\ref{thm-feb2existence}}\label{sec:defs}

In this section, we prove Theorem~\ref{thm-feb2existence}. Existence of a solution to (\ref{eq:RFKPP-general})
follows by a rather standard argument. To prove the uniqueness, we use Girsanov's theorem. In order to be able
to apply this theorem, we need to have an a priori bound showing that for any solutions to (\ref{eq:RFKPP-general})
 taking values in~$\hBI$ for all $t\ge 0$ with $R(u_0)<+\infty$, $L(u_0)>-\infty$, we have $-\infty<L(u_t)<R(u_t)<+\infty$
for all~$t\ge 0$, almost surely. 

\subsection{Existence of a solution}

We first show that (\ref{eq:RFKPP-general}) has a mild solution. 
The notion of a mild solution to (\ref{eq:RFKPP-general})
follows the  standard definition, see Walsh \cite{wal86}.  
We interpret (\ref{eq:RFKPP-general})  as a 
shorthand for the mild form,
\begin{align}
\label{eq:mild-form}
u_t(x)=&\int_{\R}G_t(x-y)u_0(y)dy+
\int_{0}^{t}\int_{{\R}}G_{t-s}(x-y)f(u_s(y))dyds  \nonumber\\
&+
\int_{0}^{t}\int_{{\R}}G_{t-s}(x-y)
 \sqrt{u_s(y)(1-u_s(y))}W(dyds) ,  
\end{align}
where $u_0(x)$ is the given initial condition.  Here, 
\[
G_t(x)=(2\pi t)^{-1/2}\exp\left(-x^2/(2t)\right),
\]
is the fundamental solution 
of the heat equation
\[
\partial_tG=\farc{1}{2}\partial_x^2G.
\]
In what follows, with some abuse of notation $\{G_t\,, t\geq 0\}$ will also denote the corresponding semigroup, that is, 
\begin{equation}
G_t\phi(x)=\int_\R G_t(x-y)\phi(y)\,,dy,\; t>0, 
\end{equation} 
for any function $\phi$ for which the above integral is well-defined.  

Almost sure existence and uniqueness of mild solutions to SPDEs
of the form 
\begin{equation}
\label{feb1220}
\partial_tu=\frac{1}{2}\partial_x^2u+f(u)+a(u)\dot{W}(t,x)
\end{equation}
is standard~\cite{wal86} when the 
coefficients are Lipschitz continuous functions of $u$.  Because  
in our case   $f(u)$ may be not Lipschitz, and $a(u)=\sqrt{u(1-u)}$  is  not Lipschitz, one needs to be slightly more 
careful. 
Solutions to~(\ref{eq:RFKPP-general}) are constructed 
as follows. Let the initial condition $u_0$ satisfy \eqref{feb1206}.  We approximate~$f(u)$ and $a(u)$ by   
Lipschitz functions $f_n(u)$ and $a_n(u)$ such that
\[
f_n(0)=f_n(1)=a_n(0)=a_n(1)=0,
\]
and construct the corresponding solutions $u^{n}_t(x)$
using the standard theory. The comparison principle implies that  
$u^n_t(x)$ take values in $[0,1]$, see \cite{shi94} and \cite{mue91s}.
The proof of Theorem~2.6 of~\cite{shi94}, on pp. 436-437 of that paper,
shows that the sequence $u^n_t(x)$ is tight. Passing to the limit~$n\to+\infty$
we obtain a mild solution $u_t(x)$ to (\ref{eq:RFKPP-general}) taking values in $[0,1]$.  
This proves existence of a solution.

\subsection{Uniqueness via the Girsanov theorem}
\label{sec:gir}

In order to prove uniqueness in law of the solution to (\ref{eq:RFKPP-general}),
we will use a version of the Girsanov theorem that will allow us to compare 
the laws of the solution $u_t(x)$ to~\eqref{eq:RFKPP-general} 
and $w_t(x)$, the solution to \eqref{eq:spde-w}, which corresponds to $f=0$ in
(\ref{eq:RFKPP-general}), 
with the same initial condition~$w_0(x)=u_0(x)$. Recall that we have set $\sigma=1$,
including in (\ref{eq:spde-w}). 
Let $\Pm_{t,u}$ be the measure induced on the canonical path space up to time $t$ by 
$u$, and $\Pm_{t,w}$ be the measure induced by $w$, also up to time $t$.  
We also define the corresponding expectations $\E_{t,u}$ and $\E_{t,w}$, and  
write $\Pm_u$ for $\Pm_{\infty,u}$, and likewise $\Pm_w$ for $\Pm_{\infty,w}$.  
We will not use the subscripts in the situations when it is clear which probability measure is used.

In~\cite{daw78}, Dawson gives a version of Girsanov's theorem which applies 
to $\Pm_{t,u}$ and $\Pm_{t,w}$. We will use its variant, Theorem IV.1.6 in~\cite{perkins}.
In such theorems, the change of measure always 
involves an exponential term which must be a martingale.  
In our situation, let
\begin{align}
\label{eq:Girsanov-u}
Z_t&=\int_{0}^{t}\int_{{\R}}
\frac{f(w_s(x))}{\sqrt{ w_s(x)(1-w_s(x))}}W(dx,ds)
  -\frac{1}{2} \int_{0}^{t}\int_{{\R}}\frac{f(w_s(x))^2}{w_s((x)1-w_s(x))}dxds.
\end{align}
Here, and elsewhere we adopt the convention in the integrands that
\[
\frac{f(u)}{\sqrt{u(1-u)}}=0\hbox{ if $u=0$ or $u=1$.}
\]
Then Girsanov's theorem for stochastic PDE \cite{daw78,perkins} says that 
\begin{equation}
\label{Girsanov1-feb12}
\frac{d\Pm_{t,u}}{d\Pm_{t,w}}=e^{Z_{t}},
\end{equation}
as long as
\begin{equation}\label{feb1230}
\int_{0}^{t}\int_{{\R}}\frac{f(u_s(x))^2}{u_s(x)(1-u_s(x))}dxds<+\infty,~~
\text{$\Pm_u$-almost surely.}
\end{equation}
In particular, if (\ref{feb1230}) holds
then (\ref{Girsanov1-feb12}) implies immediately that the solution to (\ref{eq:RFKPP-general})
is unique in law. For the moment, as we do not have any information on the support of $f(u_s(x))$, we can not conclude
that (\ref{feb1230}) holds. The bulk of the rest of this section is to show that (\ref{feb1230}) holds for any solution to~(\ref{eq:RFKPP-general})  taking values in $\hBI$ for all $t\ge 0$ and such that 
$R(u_0)<+\infty$ and~$L(u_0)>-\infty$. 

First,
we  make a much simpler observation that allows us to use Girsanov's theorem to eliminate the drift on a finite interval.
Fix and arbitrary $b>0$ and 
let $v^b$ denote a solution to a modified version of  (\ref{eq:RFKPP-general}), with the nonlinearity set to zero on the interval
$[-10b,10b]$:
\begin{equation}
\label{eq:RFKPP-general_a}
\partial_tv^b_t(x)=\farc12\partial_x^2v^b_t(x)+f(v^b_t(x)) \1_{\{x\in   (-\infty, -10b)\cup (10b,\infty)\}}+\sqrt{v^b_t(x)(1-v^b_t(x))}\dot{W}(t,x).
\end{equation}
We again write this equation in the mild form:
\begin{equation}
\label{eq:RFKPP-general_mild_a}
v^b_t(x)=G_tv^b_0(x)+\int_0^t\int_\R G_{t-s}(x-z)  f(v^b_s(z)) \1_{\{z\in  (-\infty, -10b)\cup (10b,\infty)\}}\,dz+N^b_t(x),
\end{equation}
where
\begin{equation}\label{feb2402}
N^b_t(x)=\int_0^t\int_\R G_{t-s}(x-z) \sqrt{v^b_s(z)(1-v^b_s(z))}{W}(ds,dz).
\end{equation}
Let $\Pm_{t,v^b}$ be the measure induced on the canonical path space up to time $t$ by 
$v^b$, with the corresponding expectation $\E_{t,v^b}$, and $\Pm_{\infty,v^b}$ be
$\Pm_{v^b}$. Note that by~\eqref{feb1202} we have
\begin{equation}\label{18_2_15}
\int_{0}^{t}\int_{{\R}}\frac{f(u_s(x))^2 \1_{\{x\in   (-10b,10b)\}}}{u_s(x)(1-u_s(x))}dxds\leq 20b K_f^2t 
<+\infty,~~
\text{$\Pm_u$-almost surely.}
\end{equation}
Thus we can use  Girsanov's theorem for stochastic PDE \cite{daw78,perkins} to get
\begin{equation}
\label{Girsanov1-feb12-b}
\frac{d\Pm_{t,u}}{d\Pm_{t,v^b}}=e^{Z^b_{t}},
\end{equation}
where
\begin{align}
\label{eq:Girsanov-u-b}
Z^b_t&=\int_{0}^{t}\int_{{\R}}
\frac{f(v^b_s(x)) \1_{\{x\in   (-10b,10b)\}}}{\sqrt{ v^b_s(x)(1-v^b_s(x))}}W(dx,ds)
  -\frac{1}{2} \int_{0}^{t}\int_{{\R}}\frac{f(v^b_s(x))^2 \1_{\{x\in   (-10b,10b)\}}}{v^b_s((x)1-v^b_s(x))}dxds.
\end{align}

\subsection{A bound on the front speed}

The next step is to get the following bound on the speed of the front of $u$.
\begin{lemma}
\label{lem:feb12-1}
Let $u_t(x)$ be a solution to (\ref{eq:RFKPP-general}) taking values in $\hBI$ for all $t\ge 0$ such that the initial
condition~$u_0(x)$ 
satisfies (\ref{feb1206}) with $R(u_0)\leq 0$.  
Then, for all $T> 0$, both $\sup_{t\leq T} R(u_t)$ and
$\sup_{t\leq T}L(u_t)$ are almost surely finite.
Moreover, for all $T\ge 0$
there exists $C_T>0$ so that for 
all~$b\ge 4\sqrt{T}(T\|f\|_\infty \vee 1)$ we have
\begin{equation}\label{feb1224f}
\Pm\big(\sup_{0\le t\le T}|R(u_t)-R_0|>b\big)+
\Pm\big(\sup_{0\le t\le T}|L(u_t)-L_0|>b\big)\le 
C_T\exp\Big(-\farc{b^2}{100T}\Big).
\end{equation}
\end{lemma}
An immediate consequence is 
\begin{corollary}\label{cor-feb1208}
We have, for each $T\ge 0$:
\begin{equation}\label{feb1224bis}
 \E\big[\sup_{0\le t\le T}|R(u_t)-L(u_t)|\big]<+\infty.
\end{equation}
\end{corollary}
In other words, any solution to (\ref{eq:RFKPP-general}) has an interface that has a finite length almost surely.

\subsubsection*{Bounds on the martingale with the cut-off}

The proof of Lemma~\ref{lem:feb12-1} relies on a priori bounds
on the propagation of $v^b$, solution to (\ref{eq:RFKPP-general_a}). 
First, we need to control the modulus of continuity of the martingale $N^b_t(\cdot)$ defined in (\ref{feb2402}).
\begin{lemma}
\label{lem:feb19_1}
Let $v^b_t(x)$ be a solution to (\ref{eq:RFKPP-general_a})  taking values in $\hBI$ for all $t\ge 0$, such that  the initial
condition~$v^b_0(x)$
satisfies (\ref{feb1206}) with $R(v^b_0)\leq 0$.  Then, for all $p\geq 1$, there exists $C(p)>0$ so that for 
all  $t\ge 0$, and $x,y\in [b/2, 9b]$ we have
\begin{align}\label{feb17_1}
\E\big[|N^b_t(x)-N^b_t(y)|^{2p}\big] &\le C(p)(|x-y|\wedge t^{1/2})^{p-1}t^{1/2}\\
 \nonumber 
& \times \int_\R (G_t(x-z)+G_t(y-z)) (v^b_0(z) +t \| f\|_{\infty}\1_{\{z\in  (-\infty, -10b)\cup (10b,\infty)\}})\,dz,\\
\label{feb17_2}
\E\big[|N^b_t(x)-N^b_s(x)|^{2p}\big] &\le C(p)|t-s|^{(p-1)/2}t^{1/2}\\
 \nonumber & \times \int_\R (G_t(x-z)+G_s(x-z))( v^b_0(z) + t\| f\|_{\infty}\1_{\{z\in  (-\infty, -10b)\cup (10b,\infty)\}})\,dz.
\end{align}
\end{lemma} 
\begin{proof}
The proof follows the lines of  the proof of Lemma~3.1 in~\cite{tri95}.
We only verify~\eqref{feb17_1}. Note that 
\begin{equation} 
\int_0^t\int_\R (G_{t-s}(x-z)-G_{t-s}(y-z))^2\,dz\,ds \leq C(|x-y|\wedge t^{1/2})\;\;\forall t>0, x,y\in \R.
\end{equation}
Burkholder's and  H\"older's inequalities give
\begin{equation}\label{feb17_3}
\begin{aligned}
\E\big[|&N^b_t(x)-N^b_t(y)|^{2p}\big]
\le C(p)\E\left[ \left( \int_0^t\int_\R (G_{t-s}(x-z)-G_{t-s}(y-z))^2 v^b_s(z)(1-v^b_s(z))\,dz\,ds \right)^p \right] 
\\
&\le  C(p)(|x-y|\wedge t^{1/2})^{p-1}
\E\left[\int_0^t\int_\R (G_{t-s}(x-z)-G_{t-s}(y-z))^2 \big( v^b_s(z)(1-v^b_s(z))\big)^p\,dz\,ds  \right] 
\\
&\le  C(p)(|x-y|\wedge t^{1/2})^{p-1}
\E\left[\int_0^t\int_\R (G_{t-s}(x-z)-G_{t-s}(y-z))^2  v^b_s(z)\,dz\,ds  \right]
\\
&\le  C(p)(|x-y|\wedge t^{1/2})^{p-1}
\E\left[\int_0^t (t-s)^{-1/2} \int_\R (G_{t-s}(x-z)+G_{t-s}(y-z))  v^b_s(z)\,dz\,ds  \right].
\end{aligned}
\end{equation}
We used the fact that $0\le v^b\le 1$ in the third inequality above. 
Note that 
\begin{equation}\label{feb17_4}
\begin{aligned}
\E[ v^b_s(x)]&=
G_sv^b_0(x)+\E\Big[\int_0^s\int_\R G_{s-r}(x-z)  f(v^b_r(z)) \1_{\{z\in  (-\infty, -10b)\cup (10b,\infty)\}}\,dz\,dr\Big]\\
&\leq G_sv^b_0(x)+\| f\|_{\infty} \int_0^s\int_\R G_{s-r}(x-z)   \1_{\{z\in  (-\infty, -10b)\cup (10b,\infty)\}})\,dz\,dr.
\end{aligned}
\end{equation}
We substitute this bound into the right side of~\eqref{feb17_3} and use the semi-group property of $G_t$ to get 
\begin{equation}\label{feb17_5}
\begin{aligned}
\E\big[|N^b_t(x)-N^b_t(y)|^{2p}\big]
&\!\le  C(p)(|x-y|\wedge t^{1/2})^{p-1}
\Big\{\int_0^t (t-s)^{-1/2} \Big(\int_\R (G_{t}(x-z)+G_{t}(y-z))  v^b_0(z)\,dz 
\\
\nonumber
& 
+ \int_0^s  \| f\|_{\infty} \int_\R (G_{t-r}(x-z)+G_{t-r}(y-z))  \1_{\{z\in  (-\infty, -10b)\cup (10b,\infty)\}})\,dz
\,  dr\Big)\,ds \Big\}
\\
&\le  C(p)(|x-y|\wedge t^{1/2})^{p-1}
  t^{1/2} \left(\int_\R (G_{t}(x-z)+G_{t}(y-z))  v^b_0(z)\,dz\right.
\\
\nonumber
& \quad
\left. +    \| f\|_{\infty} \int_0^t \int_\R (G_{t-r}(x-z)+G_{t-r}(y-z))  \1_{\{z\in  (-\infty, -10b)\cup (10b,\infty)\}})\,dz
\,  dr \right) .
\end{aligned}
\end{equation}
Since $x,y \in (b/2, 9b)$ and $z\geq 10b$ we have 
\begin{equation}
\int_{z\geq 10b} G_{r}(x-z)\,dz \leq \int_{z\geq 10b}  G_t(x-z)\,dz , ~~\forall x\in (b/2, 9b), 0\leq r\leq t, 
\end{equation}
and thus we get 
\begin{equation}\label{feb17_6}
\begin{aligned}
\E\big[|N^b_t(x)-N^b_t(y)|^{2p}\big]
&\le  C(p)(|x-y|\wedge |t-s|^{1/2})^{p-1}\\
&\times 
  t^{1/2} \int_\R (G_{t}(x-z)+G_{t}(y-z)) \Big( v^b_0(z)+t\| f\|_{\infty}  \1_{\{z\in  (-\infty, -10b)\cup (10b,\infty)\}}\Big)   \,dz,
\end{aligned}
\end{equation}
which is \eqref{feb17_1}.
The proof of \eqref{feb17_2} goes along similar lines. 
\end{proof}

A corollary of Lemma~\ref{lem:feb19_1} is a bound on the size of $N_b^s(x)$.
\begin{lemma}
\label{lem:tr_3.1}
Let $v^b_t(x)$ be a solution to (\ref{eq:RFKPP-general_a}),  taking values in $\hBI$ for all $t\ge 0$, and the initial
condition~$v^b_0(x)$
satisfies (\ref{feb1206}) with $R(v^b_0)\leq 0$.  Then, for all $t> 0$, 
there exists $C$ such that 
\begin{align}\label{feb17_8}
\Pm &\big(|N^b_s(x)| \geq \eps\;\text{for some}\; x\in (b/2, 9b), s\in [0,t]  \big)\\
\nonumber
&\le C\eps^{-20} (t\vee t^{22})    
 \int_\R \int_\R G_t(x-z) \Big(v^b_0(z) +t \| f\|_{\infty}\1_{\{z\in  (-\infty, -10b)\cup (10b,\infty)\}})\,dz \1_{\{x\in (b/2,9b)\}}\Big) \,dx.
\end{align}
\end{lemma} 
\begin{proof}
The proof goes exactly as the second part of the proof of Lemma~3.1 in~\cite{tri95} 
(on p.~295) while taking   $v^b_0(z) +t \| f\|_{\infty}\1_{\{z\in (-\infty, -10b)\cup (10b,\infty)\}}$ instead of
$f$ and $  (b/2,9b)$ instead of $(A,\infty)$ there. 
\end{proof}

\subsubsection*{The support of the solution with a cut-off}

Now, we prove the following lemma. 
\begin{lemma}
\label{lem:18_2_2}
Let $v^b_t(x)$ be a solution to (\ref{eq:RFKPP-general_a})  taking values in $\hBI$ for all $t\ge 0$ such  that  the initial
condition~$v^b_0(x)$
satisfies  (\ref{feb1206}) with $R(v^b_0)\leq 0$.  Then, for all $t> 0$
there exists $C_t>0$ so that for 
all~$b\ge 4\sqrt{t}(t \|f\|_\infty \vee 1)$ we have
\begin{equation}\label{feb1224}
\Pm\big(\sup_{0\le s\le t}\sup_{x\in [b,2b]} v^b_s(x)>0\big)\le 
C(t, \| f\|_{\infty})\exp\Big(-\farc{b^2}{50t}\Big).
\end{equation}
\end{lemma} 
\begin{proof}

We will 
follow the proof  of Proposition~3.2 in~\cite{tri95}. 
Let us take a function~$\psi\in L^1(\Rm)\cap C(\Rm)$ such that $0\le\psi(x)\le 1$
for all $x\in\Rm$ and  
$\{x: \psi(x)>0\} = (0,b)$, and set $\psi_b(x)=\psi(x-b)$. For simplicity of
notation, we define 
\[ 
\langle h,g\rangle = \int_{\R} h(x) g(x)\,dx  
\]
for any functions $h,g$ such that the integral above exists. 

Fix $t>0$ and let $\phi^\lambda_s(x)$,  $0\le s\le t$, $x\in \R$  
be the unique non-negative bounded solution to the backward in time problem
\begin{equation}
\label{eq:12_1}
-\partial_s \phi_s^\lambda = \frac{1}{2}\Delta \phi^\lambda_s - \frac{1}{4}(\phi^\lambda_s)^2 +\lambda\psi_b,
\end{equation} 
with the terminal condition $\phi^\lambda_t(x) \equiv 0. $ 
A similar equation to (\ref{eq:12_1}) but
with different function $\psi_b$ in the right side 
appears in the proof of Proposition~3.2 in~\cite{tri95}. As $\psi_b(x)\ge 0$ for all
$x\in\Rm$, the maximum principle implies existence of the solution to~\eqref{eq:12_1}
and that $\phi_s^\lambda(x)\ge 0$ for all $0\le s\le t$ and $x\in\Rm$.
The maximum principle also implies that
\[   
\phi^\lambda_{s}(x) \leq \lambda
\int_0^{t-s} \int G_{r}(x-y) \psi_b(y)dy \,dr,\;\;s\leq t,     
\] 
and thus $\phi^\lambda_s(x)$ is integrable for all $0\le s\le t$. Next, 
note that the function
\[
\zeta_t(x)=\left\{\begin{array}{l}  \farc{\alpha}{(x-b)^2}, x<b,
\\
\farc{\alpha}{(x-2b)^2}, x>2b,
\end{array}\right.
\]
satisfies, in the region $x<b$, where $\psi_b(x)\equiv 0$:
\[
\begin{aligned}
\partial_t\zeta&-\farc{1}{2}\Delta\zeta+\frac{1}{4}\zeta^2-\lambda\psi_b=
-\frac{1}{2}\farc{2\cdot 3\alpha}{(x-b)^4}
+\frac{\alpha^2}{4(x-b)^4}
=
\farc{\alpha(\alpha -12)}{(x-b)^4}\ge 0,
\end{aligned}
\]
provided that we take $\alpha\ge 12$. As $\zeta_t(x)=+\infty$ at $x=b$, the
maximum principle implies that, for $\alpha$ sufficiently large, we have
\begin{equation}
\label{eq:12_2}
 \phi^\lambda_s (x)\leq 
 \farc{\alpha}{(b-x)^{2}},\;\;\text{ for all $x<b$, $s\leq t$, 
 and $\lambda>0$.} 
\end{equation}  
Similarly, again for $\alpha$ large enough, we get 
\begin{equation}
\label{eq:12_2a}
 \phi^\lambda_s (x)\leq 
 \farc{\alpha}{(2b-x)^{2}},\;\;\text{ for all $x>2b$, $s\leq t$, 
 and $\lambda>0$.} 
\end{equation}  
Now, given any $b\ge 4t^{1/2}$, we may use the fundamental solution for the heat equation
on the half-lines~$x<b-t^{1/2}$, $x>2b+t^{1/2}$  together with the upper bound in (\ref{eq:12_2}) on
$\phi_s^\lambda(x)$ at $x=b-t^{1/2}$, and $x=2b+t^{1/2}$ to conclude that
there exists $\alpha_1>0$  
such that 
\begin{equation}
\label{eq:12_3}
 \phi^\lambda_s(x) \leq \frac{\alpha_1}{t}  
 \exp{\Big( -\frac{(b-x)^2}{20t}\Big) }    ,\;\;\text{
for all  $b\geq 4 t^{1/2}$, $x<b-2t^{1/2}$,  $s\leq t$, and $\lambda>0$,} 
\end{equation}  
and
\begin{equation}
\label{eq:12_3bis}
 \phi^\lambda_s(x) \leq \frac{\alpha_1}{t}  
 \exp{\Big( -\frac{(2b-x)^2}{20t}\Big) }    ,\;\;\text{
for all  $b\geq 4 t^{1/2}$,  $x>2b+2t^{1/2}$, $s\leq t$, and $\lambda>0$.} 
\end{equation}

Next, by It\^o's formula, 
we get, for any $0\le s\le t$: 
\begin{align}
\exp\Big( &-\langle v^b_s\,, \phi^\lambda_s \rangle
-\lambda \int_0^s \langle v^b_{s'}\,,\psi_b\rangle\,ds' \Big)
=  \exp\big( -\langle v^b_0\,, \phi^\lambda_0 \rangle \big)
+ \int_0^s \!\!
\exp\Big( -\langle v^b_{s'}\,, \phi^\lambda_{s'} \rangle
-\lambda \int_0^{s'} \langle v^b_r\,,\psi_b\rangle\,dr \Big) \nonumber\\
\nonumber
&\times\Big(
\langle v^b_{s'}, -\partial_s \phi_{s'}^\lambda- \frac{1}{2}\Delta \phi^\lambda_{s'}  
- \lambda\psi_b\rangle  
-\langle f(v^b_{s'})  \1_{(-\infty,-10b)\cup (10b,\infty)},\phi^\lambda_{s'}\rangle +
\frac{1}{2}\langle v^b_{s'}(1-v^b_{s'}),(\phi^\lambda_{s'})^2\rangle
  \Big)\,ds' \\
  \nonumber
  &
 + M^{\phi^\lambda, \psi_b}_s\,, 
\end{align}
where $s\mapsto M^{\phi^\lambda, \psi_b}_s\,, s\leq t,$ is a local martingale. In fact,  $M^{\phi^\lambda, \psi_b}$ is 
a square integrable martingale: this follows easily from integrability 
of $(\phi^\lambda)^2$. Then we get 
\begin{align}\label{feb1322}
\exp\Big( &-\langle v^b_s\,, \phi^\lambda_s \rangle
-\lambda \int_0^s \langle v^b_{s'}\,,\psi_b\rangle\,ds' \Big)
=  \exp(-\langle v^b_0\,, \phi^\lambda_0 \rangle  ) 
\mbox{}+ \int_0^s 
\exp\Big( -\langle v^b_{s'}\,, \phi^\lambda_{s'} \rangle-
\lambda \int_0^{s'} \langle v^b_r\,,\psi_b\rangle\,dr \Big) \nonumber\\
&\times \big(
\langle -f(v^b_{s'})  \1_{(-\infty,-10b)\cup (10b,\infty)},  \phi_{s'}^\lambda\rangle  
+\langle -\frac{1}{4}v^b_{s'}+\frac{1}{2}v^b_{s'}(1-v^b_{s'}),(\phi^\lambda_{s'})^2\rangle
  \big)\,ds' + M^{\phi^\lambda, \psi_b}_s.
\end{align}
Note that (\ref{eq:12_2}) implies that for $b>R_0$ we have a uniform bound 
\begin{equation}\label{feb1320}
|\langle v^b_0\,, \phi^\lambda_0\rangle|\le c_0,
\end{equation}
with a constant $c_0$ that does not depend on $\lambda$.  
Now we define the stopping times
\[
\tau_b=\inf\Big\{t\geq 0: \exists x\in [b/2,3b]\text{ s.t. }
v^b_t(x)\geq \frac{1}{2}  \Big\},\;\; \rho_b=\inf\{ t\geq 0: \langle v^b_t,\psi_b\rangle>0\}. 
\]
Note that we have
\begin{equation}\label{feb1328}
\langle v^b_{t\wedge \tau_b}\,, \phi^\lambda_{t\wedge \tau_b} \rangle 
+\lambda \int_0^{t\wedge \tau_b} \langle v^b_s\,,\psi_b\rangle\,ds\to+\infty
\hbox{ as $\lambda\to+\infty$,}
\end{equation}
almost surely on the event $\{\rho_b<t\wedge\tau_b)$, thus 
\begin{align}\label{feb1326}
\Pm(\rho_b<t\wedge\tau_b)\le \lim_{\lambda\to+\infty}
\E\Big[ 1-\exp\Big(&-\langle v_{t\wedge \tau_b}\,, \phi^\lambda_{t\wedge \tau_b} \rangle 
-\lambda \int_0^{t\wedge \tau_b} \langle v_s\,,\psi_b\rangle\,ds \Big)\Big]
\end{align}
On the other hand, taking the expectation in (\ref{feb1322}) with $s=t\wedge \tau_b$, we get
\begin{align}
\E\Big[ 1-\exp\Big(&-\langle v^b_{t\wedge \tau_b}\,, \phi^\lambda_{t\wedge \tau_b} \rangle 
-\lambda \int_0^{t\wedge \tau_b} \langle u_s\,,\psi_b\rangle\,ds \Big)\Big]
\leq  \E\big[1-\exp\big( -\langle v^b_0\,, \phi^\lambda_0 \rangle \big)\big]\nonumber\\
&\mbox{}+ \E\Big[ \int_0^t \left( \|f\|_\infty\langle \1_{(-\infty,-10b)\cup (10b,\infty)},  \phi_{s}^\lambda\rangle+
\langle \frac{1}{4}v^b_s \1_{ (-\infty,b/2)\cup (3b,\infty)},(\phi^\lambda_s)^2\rangle \right)
  \,ds\Big].\label{feb1324}
\end{align}
Note that for each $0\le s\le t$ and $x\in\Rm$ the family 
$\phi_s^\lambda(x)$ is increasing in $\lambda$.
Moreover, for $s<t$ and~$x>b$
we have $\phi_s^\lambda(x)\to+\infty$ as $\lambda\to+\infty$, while for $x<b$,
the limit $\phi_s^\infty(x)$ is finite because of~(\ref{eq:12_2}). 
Passing to the 
limit $\lambda\to+\infty$ in (\ref{feb1324}), using the  bound
in (\ref{feb1326}) and since~$v^b_s(x)\leq 1$ for all $s\geq 0, x\in \R$, we get 
\begin{align}\label{feb1330}
\Pm(\rho_b<t\wedge\tau_b)\le
\E\big[&1-\exp\big( -\langle v^b_0\,, \phi^\infty_0 \rangle \big)\big]    \\
&+ \int_0^t \left(
\|f\|_\infty\langle \1_{(-\infty,-10b)\cup (10b,\infty)},  \phi_{s}^\infty\rangle
+
\frac{1}{4}\langle \1_{ (-\infty,b/2)\cup (3b,\infty)},(\phi^\infty_s)^2\rangle
 \right)
  \,ds .
\nonumber
\end{align}
Recalling (\ref{eq:12_3})-(\ref{eq:12_3bis}),  we have 
\begin{align}\label{feb1330bis}
\Pm(\rho_b<t\wedge\tau_{b})&\le \frac{C}{t}\int_{-\infty}^0 e^{-(b-x)^2/(20t)}dx\\
\nonumber
&\qquad + \frac{\|f\|_\infty}{t} \int_{0}^{t}\left(\int_{10b}^{\infty} e^{-(2b-x)^2/(20t)}\,dx + 
  \int_{-\infty}^{-10b} e^{-(b-x)^2/(20t)}\,dx \right)ds \\
\nonumber 
&\qquad\qquad+\frac{C}{t^2}\int_0^t\left(\int_{-\infty}^{b/2} e^{-(b-x)^2/(10t)}
 dx  + \int_{3b}^{\infty} e^{-(2b-x)^2/(10t)}\,dx\right)ds \\
 \nonumber
&\le \frac{C}{b}\exp\left(-\frac{b^2}{40t}\right)\nonumber
+
\frac{C\|f\|_\infty}{b}\exp\Big(-\frac{b^2}{t}\Big)
\le  \frac{C(\|f\|_\infty+1)}{t^{1/2}}\exp\Big(-\frac{b^2}{40t}\Big).\nonumber
\end{align}
We used the assumption that $R_0=0$ in the first term in the right side above.  To estimate the integrals in \eqref{feb1330bis}, 
we used the standard Gaussian estimate
\begin{equation*}
\int_y^\infty\exp(-x^2/2)dx\leq y^{-1}\exp(-y^2/2)
\end{equation*}
along with a few changes of variables.

Now we need to estimate
\begin{align}
\nonumber
\Pm(\tau_b\leq t)&= \Pm\left(\exists x\in  [b/2,3b],  s\leq t: G_sv^b_0(x)+\int_0^s\left(\int_{10b}^\infty
+\int_{-\infty}^{-10b}\right) G_{s-r}(x-z)  f(v^b_r(z))\,dz\,dr \right.\\
\label{18_02_1}
 &\;\;\;\;\;\;\;\;\left.+N^b_s(x)\geq 1/2\right).  
\end{align}
It is easy to check that since $b\geq 4\sqrt{t}(t\|f\|_{\infty}\vee 1)$
\begin{align}
 G_sv^b_0(x)&\leq    \int_{-\infty}^0  G_s(x-z)\,dz \leq  \int_{-\infty}^0  G_t(x-z)\,dz\leq  \int_{-\infty}^0  G_t(b/2-z)\,dz 
 \\
 \nonumber 
 &\leq \int_{-\infty}^0  G_1(2-z)\,dz \leq 1/10,\;\forall s\leq t, x\in [b/2, 3b]. 
\end{align} 
Similarly, we have
\begin{align}
t\int_{10b}^{\infty}  G_{s-r}(z-x) f(v^b_r(z) \,dz&\leq  t\|f\|_{\infty}  \int_{10b}^\infty  G_t(z-3b)\,dz
\leq  t\|f\|_{\infty}  \int_{\frac{7b}{\sqrt{t}}}^\infty  G_1(z)\,dz,\\
\nonumber
&\leq  t\|f\|_{\infty}  \int_{ 28(t\|f\|_{\infty}\vee 1)}^\infty  G_1(z)\,dz\leq 0.05,\;\forall r\leq s\leq t, x\in [b/2, 3b]. 
\end{align} 
and
\begin{align}
t\int_{-\infty}^{-10b}  G_{s-r}(z-x) f(v^b_r(z) \,dz&\leq  t\|f\|_{\infty} \int_{-\infty}^{-10b}  G_t(z-b/2)\,dz
\leq  t\|f\|_{\infty}  \int_{\frac{10b}{\sqrt{t}}}^\infty  G_1(z)\,dz,
\\ \nonumber &
\leq  t\|f\|_{\infty}  \int_{40(t|f\|_{\infty}\vee 1)}^\infty  G_1(z)\,dz\leq 0.05,\;\forall r\leq s\leq t, x\in [b/2, 3b]. 
\end{align}
Altogether substituting the last inequalities into~\eqref{18_02_1} we get 
\begin{align}
\label{18_02_2}
\Pm(\tau_b\leq t)&\leq \Pm\left(\exists x\in  [b/2,3b],  s\leq t:N^b_s(x) \geq 0.3\right)\\
\nonumber
&\le C\cdot(t\vee t^{22})    
 \int_\R \int_\R G_t(x-z) (v^b_0(z) +t \| f\|_{\infty}\1_{\{z\in (-\infty, -10b)\cup (10b,\infty)\}})\,dz \1_{\{x\in (b/2,9b)\}}) \,dx
\\
\nonumber 
&\leq 
C(t, \| f\|_{\infty})\exp\Big(-\farc{b^2}{50t}\Big),\; \forall t>0, x\in [b/2, 3b],
\end{align}
where the second inequality follows by Lemma~\ref{lem:tr_3.1} and in the last one  we used simple Gaussian bounds.
By combining \eqref{18_02_2} with~\eqref{feb1330bis} we are done.  

\end{proof}

\subsubsection*{The proof of Lemma~\ref{lem:feb12-1}}

Now we are ready to prove Lemma~\ref{lem:feb12-1}.
Note that Lemma~\ref{lem:18_2_2}
implies a similar result for  $u_t(x)$. 
\begin{lemma}
\label{lem:18_2_3}
Let $u_t(x)$ be a solution to (\ref{eq:RFKPP-general}) taking values in $\hBI$ for all $t\ge 0$ such that the initial
condition~$u_0(x)$
satisfies (\ref{feb1206}) with $R(u_0)\leq 0$.  Then, for all $T> 0$
there exists $C_T>0$ so that for 
all~$b\ge 4\sqrt{T}(T \|f\|_\infty \vee 1)$ we have
\begin{equation}\label{feb1224_c}
\Pm\Big(\sup_{0\le t\le T}\sup_{x\in [b,2b]} u_t(x)>0\Big)\le 
C(T, \| f\|_{\infty})\exp\Big(-\farc{b^2}{50T}\Big).
\end{equation}
\end{lemma}
\begin{proof}
By Girsanov's theorem we have
\begin{align}\nonumber
\Pm_u\Big(\sup_{0\le t\le T}\sup_{x\in [b,2b]} u_t(x)>0\Big)&\le 
\E_{v^b}\big[e^{Z^{b}_T} \1_{\{\sup_{0\le t\le T}\sup_{x\in [b,2b]} v^b_t(x)>0\}}\big]
\\
\label{feb1224_d}
&\le 
\E_{v^b}\big[e^{2Z^{b}_T}\big] \Pm_{v^b}\Big(\sup_{0\le t\le T}\sup_{x\in [b,2b]} v^b_t(x)>0\Big),
\end{align}
where $Z^b$ was defined in \eqref{eq:Girsanov-u-b}.
Note that \eqref{18_2_15} holds also $\Pm_{v^b}$-a.s., 
thus from \eqref{eq:Girsanov-u-b} we can easily get 
\begin{equation}
\E_{v^b}\big[e^{2Z^{b}_T}\big] \le e^{40bK^{2}_{f}T},
\end{equation}
and combining this with \eqref{feb1224_d} and Lemma~\ref{lem:18_2_2} we obtain~\eqref{feb1224_c}.
\end{proof}

Now, the conclusions of Lemma~\ref{lem:feb12-1} follow essentially immediately. 
The bound (\ref{feb1224f}) on 
\[
\Pm\big(\sup_{0\le t\le T}|R(u_t)-R_0|>b\big)
\]
in 
Lemma~\ref{lem:feb12-1} is a simple consequence of Lemma~\ref{lem:18_2_3}.
The  finiteness of $\sup_{t\leq T} R(u_t)$ follows from~(\ref{feb1224f}). The corresponding bounds on $L(u_t)$
follow by repeating the arguments used in the proof of  
Lemmas~\ref{lem:feb19_1}--\ref{lem:18_2_3} for $1-u(-x)$ instead of~$u(x)$.  

\subsubsection*{Uniqueness of the solution  }

So far, we have shown that both $R(u_t)$ and $L(u_t)$ are $\Pm_u$-a.s. finite for any solution
to (\ref{eq:RFKPP-general})  taking values in $\hBI$ for all $t\ge 0$ such that
 the initial
condition~$u_0(x)$
satisfies (\ref{feb1206}).  As a consequence,  (\ref{feb1230}) holds for any such solution to (\ref{eq:RFKPP-general}).
As we have discussed in Section~\ref{sec:gir}, it follows that 
we may apply Girsanov's theorem to immediately deduce
uniqueness in law of the solution to~(\ref{eq:RFKPP-general}) that satisfies the above conditions. 

\subsection{Existence of the speed}

The last ingredient in the proof of Theorem~\ref{thm-feb2existence}
is the existence of the speed. 
\begin{lemma}\label{lem-feb1204}
There exists a deterministic constant $V(\sigma)\in(-\infty,+\infty)$ 
so that the limit
\begin{equation}\label{eb1234}
V(\sigma)=\lim_{t\to+\infty}\farc{R(u_t)}{t}
\end{equation}
exists almost surely.
\end{lemma}
\begin{proof}
The proof goes along the lines of the proof of the corresponding result in~\cite{cd05}. 
First, we show that the  limit $V(\sigma)$  in~\eqref{eb1234} exists and 
 $V(\sigma)<\infty$. Let us
set $b(m)=R(u_m)$, for~$m=0,1,2,\ldots$, and note that by Corollary~\ref{cor-feb1208} 
we have 
\begin{equation}
\E\left[ (b(1)-b(0))_+\right]<\infty. 
\end{equation}
Then, as in the proof of Lemma~5.1 in~\cite{cd05} we can use the subadditive ergodic theorem to deduce that there exists a constant
$c(\sigma)\in [-\infty,\infty)$, such that 
\begin{equation}\label{eeq:12_6}
\lim_{m\to+\infty}\farc{b(m)}{m}=c(\sigma). 
\end{equation}
Using Lemma~\ref{lem:feb12-1},  we get  (see  Lemma~5.3 in~\cite{cd05} for the same argument) that for all $m=1,2,\ldots$
\begin{equation}
\Pm\Big( \sup_{0\leq s\leq 1}\left\{b(s+m)-b(m), b(m+1)-b(s+m)\right\} > \sqrt{m}\Big) \leq C(\sigma) \exp( -m/50). 
\end{equation}
Then  by the Borel-Cantelli lemma we get that in fact, 
\begin{equation}\label{eeq:12_7}
\lim_{t\to+\infty}\farc{b(t)}{t}=c(\sigma). 
\end{equation}
and thus $V(\sigma)=c(\sigma)<\infty$. 

To show that  $V(\sigma)>-\infty$, one needs to consider equation for $1-u_t(-x)$ and repeat the above argument. 
 
\end{proof}

\section{The interface in the voter model}\label{sec:voter}

Girsanov's theorem  connecting solutions to the rescaled
equation (\ref{eq:spde-v_f}) and to the voter model~(\ref{eq:spde-w}) not only allows us to 
deduce uniqueness in the law for the solutions to the former problem but also obtain the asymptotics
on their front speed in Theorem~\ref{th:2}.
As a preliminary step, in this section, we make some
observations about the latter. 
To begin, we rephrase  Lemma~4.2(a)  of~\cite{tri95}, putting it into a form more directly useful for our 
purposes.  Let $w_t(x)$ be the solution to~\eqref{eq:spde-w} with an initial 
condition~$w_0(x)$ satisfying (\ref{feb1206}). Recall that we denote by
$\Pm_w$  the measure induced on the canonical 
path space~$C([0,+\infty);C(\R))$  
by~$w$, and by $\E_w$ we denote the corresponding expectation.
Recall that two random processes $X_t$ and $Y_t$ are said to be coupled
if they can be defined on the same probability space. We assume throughout 
the rest of the paper that $f$ satisfies assumption~(\ref{feb1202a}). 
\begin{lemma}
\label{lemma:tri95}
Given $\varepsilon>0$, there exists $T_\eps>0$ such that for 
all $T\geq T_\eps$ there is a coupling of processes~$(w_t,B_t:t\geq0)$ 
where $B$ a standard Brownian motion started at 0, such that
\[
\Pm_w\Big(\sup_{0\leq t\leq T}\big|R(w_{t})-B_t\big|
 \vee \big|L(w_{t})-B_t\big|\geq T^{1/2}\varepsilon\Big)
\leq\varepsilon.  
\]
\end{lemma}
The following lemma shows that another good measure of the location of the
interface is  
\begin{align}
M_t:= \int_0^t \int_{{\R}}\sqrt{ w_s(x)(1-w_s(x))}W(dx,ds).\;\; 
\end{align}

\begin{lemma}
\label{lemma:12_1}
Let $B$ be the Brownian motion from Lemma~\ref{lemma:tri95}. 
Given $\varepsilon>0$, there exists $T_\eps>0$ such that for 
all $T\geq T_\eps$ we have
\[
\Pm_w\Big(\sup_{0\leq t\leq T}\big|M_t-B_t\big|
  \geq 4 T^{1/2}\varepsilon\Big)
\leq\varepsilon.  
\]
\end{lemma}
\begin{proof}
By Lemma \ref{lemma:tri95},  
\begin{equation}
\label{eq:def-2}
\Xi(w_t):=\int_{-\infty}^{0}\big[w_t(x)-1\big]dx+\int_{0}^{\infty}w_t(x)dx.
\end{equation}
is an almost surely finite  
functional of $w_t$.
As $w_t(x)=1$ for $x<L(w_t)$ and~$w_t(x)=0$ for~$x>R(w_t)$, we have
\[
\Xi(w_t)=\int_{L(w_t)\wedge0}^{0}[w_t(x)-1]dx+\int_{0}^{R(w_t)\vee0}w_t(x)dx,
\]
thus
\begin{align*}
L(w_t)&=\int_{L(w_t)\wedge0}^{0}[-1]dx+\int_{0}^{L(w_t)\vee0}dx
\leq\int_{L(w_t)\wedge0}^{0}[w_t(x)-1]dx+\int_{0}^{R(w_t)\vee0}w_t(x)dx=\Xi(w_t),
\end{align*}
and likewise
\begin{align*}
R(w_t)&=\int_{R(w_t)\wedge0}^{0}[-1]dx+\int_{0}^{R(w_t)\vee0}dx
\geq\int_{L(w_t)\wedge0}^{0}[w_t(x)-1]dx+\int_{0}^{R(w_t)\vee0}w_t(x)dx=\Xi(w_t).
\end{align*}
We conclude that
\begin{equation}
\label{eq:L-Xi-R}
L(w_t))\leq\Xi(w_t)\leq R(w_t).
\end{equation}

Next, let $\theta(x)$ be a smooth monotonically decreasing function such 
that $\theta(x)=1$ for $x<-2$ and~$\theta(x)=0$ for $x>-1$, and 
set $\theta_n(x)=\theta(nx)$. Then for
\[
\zeta_n(x):=w_t(x)-\theta_n(x)
\]
we have
\[
\Xi(w_t)=\lim_{n\to\infty}\Xi_n(t),~~\Xi_n(t)=\int_{-\infty}^\infty \zeta_n(x)dx.
\]
The function $\zeta_n(t,x)$ satisfies
\begin{equation}
\label{eq:spde-zeta}
\partial_t\zeta_n=\frac{1}{2}\partial_x^2\zeta_n+\frac{1}{2}\partial_x^2\theta_n
+\sqrt{w(1-w)}\dot{W}(t,x).
\end{equation}
Integrating in $t$ and $x$ gives
\begin{equation}
\label{eq:SDE-Xin}
\Xi_n(w_t)=\Xi_n(w_0)+\int_{0}^{t}\int_{{\R}}\sqrt{w_s(y)(1-w_s(y))}W(dyds).
\end{equation}
Passing to the limit $n\to+\infty$, we arrive at
\begin{equation}
\label{eq:SDE-Xi}
\Xi(w_t)=\Xi(w_0)+\int_{0}^{t}\int_{{\R}}\sqrt{w_s(y)(1-w_s(y))}W(dyds)
=\Xi(w_0)+M_t.
\end{equation}
As $\Xi(w_0)<+\infty$ and is not random, 
the conclusion of the present lemma follows from Lemma~\ref{lemma:tri95} by 
taking~$T_\eps$ 
sufficiently large. 
\end{proof}

For any metric space $\bE$, we denote by $\bD_{\bE}$ the space of c\`{a}dl\`{a}g functions $[0,\infty)\to \bE$ equipped with the Skorohod topology. Define
the rescaled functionals 
\[
L_t^a=\frac{1}{a}L{(w_{a^2t})},~R_t^a=\frac{1}{a}R{(w_{a^2t})},~~M_t^a=\farc{1}{a}M_{a^2t}\,.
\]
As a consequence of Lemmas~\ref{lemma:tri95} and~\ref{lemma:12_1}, we conclude
that 
\[
(L^a, R^a, M^a) \Rightarrow (B,B,B)\; \text{in} \; \bD_{\R^3},\;\; \text{as}\; a\rightarrow\infty,
\]
where $B$ is a standard Brownian motion starting at $0$ and $\Rightarrow$ denotes convergence in law. 

As in the application of the Girsanov theorem
in the proof of Theorem~\ref{thm-feb2existence},
we will make use of the functionals
\begin{align}
A^f_t&:= 
\int_0^t \int_{{\R}}\frac{f(w_s(x))^2}{w_s(x)(1-w_s(x))}\,dx\,ds,\;\; \\
M^f_t&:= \int_0^t \int_{{\R}}\frac{f(w_s(x))}{\sqrt{ w_s(x)(1-w_s(x))}}W(dx,ds),\end{align}
and their rescaled versions
\[
M^{f,a}_t= \frac{1}{a}M_{a^2 t},~~   A^{f,a}_{t}= \frac{1}{a^2}A^f_{a^2t},~~a>0. 
\]
The difference in the scaling of these two functionals comes from the fact that
$M_t$ is, roughly, a Brownian motion on large time scales, 
and $A_t$ is deterministic to the leading order on large time scales. Note that both $A_t$ and $M_t$ are almost surely
finite if $f$ satisfies assumption (\ref{feb1202}), 
since the interface of $w_t$ has a finite length almost surely. However, we will need the stronger assumption~(\ref{feb1202a})
in Lemma~\ref{lem6_1} below. 

Let us now recall Theorem 1 of  \cite{mt97}.
\begin{theorem}[\cite{mt97}]
\label{th:mt-th1}
There exists a unique stationary measure $\mu$ on $\mathcal{C}_I$ for (\ref{eq:spde-w}).
Furthermore, for each $u_0\in\mathcal{C}_I$,  the law of 
$w_t(x+L_t)$ converges in total variation to $\mu$ as
$t\to\infty$.   In addition, the moment of the width of the interface
$\E_{w,st}[R(f)-L(f)]^p\mu(df)$ is finite if $0\leq p<1$, 
and infinite for~$p\geq1$.  
\end{theorem}
The following estimate is a consequence of the second part of Theorem~\ref{th:mt-th1}.
\begin{lemma}
\label{lem:feb19_2}
For any $\eta\in (0,1]$, we have
\begin{equation}
\label{eq:feb19_1}
\E_{w,st}\left[\int_{{\R}}(w(x)(1-w(x)))^\eta\,dx\right]<\infty.
\end{equation}
\end{lemma}
Note that this result fails at $\eta=0$: according to Theorem~\ref{th:mt-th1},
the length of the interface has an infinite expectation under the stationary distribution of $w$.

\begin{proof}
For $\eta=1$ the result is known (see Lemma~2.1(a) in~\cite{tri95}), so we assume that $\eta\in (0,1)$. 
Let $\ell$ be the length of the interface of $w$ under the stationary distribution. 
By applying H\"older's and Young's inequalities we get 
\[
\begin{aligned}
\E_{w,st}\Big[\int_{{\R}}(w(x)(1-w(x)))^\eta\,dx\Big]&\leq 
\E_{w,st}\Big[\Big(\int_{{\R}}(w(x)(1-w(x)))\,dx\Big)^\eta \ell^{1-\eta}\Big]\\
&\leq 
C_\alpha\E_{w,st}\Big[\Big(\int_{{\R}}(w(x)(1-w(x)))\,dx\Big)^{\alpha\eta}\Big] + 
C_\alpha\E_{w,st}\Big[ \ell^{\frac{\alpha(1-\eta)}{\alpha-1}}\Big],
\end{aligned}
\]
for any $\alpha>1$. We take $\alpha=2/\eta$ and  get 
\[
\E_{w,st}\Big[\int_{{\R}}(w(x)(1-w(x)))^\eta\,dx\Big]\leq 
C_\alpha\E_{w,st}\Big[\Big(\int_{{\R}}(w(x)(1-w(x)))\,dx\Big)^{2}\Big] +C_\alpha \E_{w,st}\big[ \ell^{\gamma}\big],
\]
with $\gamma={(1-\eta)}/{(1-\eta/2)}$.
Since $\gamma<1$,  by Theorem~\ref{th:mt-th1} we get $ \E_{w,st}[ \ell^\gamma] <\infty$. 
In addition, Lemma~2.1(d) in~\cite{tri95} implies that
\[
 \E_{w,st}\Big[\Big(\int_{{\R}}(w(x)(1-w(x)))\,dx\Big)^{2}\Big] <\infty,
 \]
and we are done. 
\end{proof}

\begin{lemma}
\label{lem6_1}
Let $f$ satisfy assumption~(\ref{feb1202a}), then we have convergence in law 
\begin{equation}
(M^{f,a}, A^{f,a}) \Rightarrow \{B^f_t,Dt),\; t\geq 0\},
\end{equation} 
in $\bD_{\R^2}$, as $a\rightarrow \infty.$  Here $\{B^f_t\,, t\geq 0\}$ is a 
Brownian motion with variance $D$ 
\begin{equation}\label{c2f-def}
D = \E_{w,st}\left[\int_{{\R}}\frac{f(w(x))^2}{w(x)(1-w(x))}\,dx\right]<\infty. 
\end{equation}
\end{lemma}
Note that $D<+\infty$ because of Lemma~\ref{lem:feb19_2} and assumption (\ref{feb1202a}) on $f$. 
\begin{proof}
Since $w$ has a unique stationary distribution on the space $\mathcal{C}_I$
of continuous functions $h$ such that $-\infty<L(h)<R(h)<+\infty$, 
by the ergodic theorem we have 
\begin{equation}
\label{6_2}
\lim_{a\rightarrow\infty} a^{-2} A^{f}_{a^2t} = 
t \E_{w,st}\left[\int_{{\R}}\frac{f(w(x))^2}{w(1-w(x))}\,dx\right]=Dt, 
\end{equation}
uniformly on compact sets in $t$.
Recall that $\E_{w,st}$ denotes the expectation with respect to the
stationary measure of $w$ on $\mathcal{C}_I$.
Since
\begin{equation}
M^{f,a}_t=\tilde B_{A^{f,a}_t},\;t\geq 0,
\end{equation}
for some standard Brownian motion $\tilde B$, it follows from (\ref{6_2}) that 
\begin{equation}
\label{16_12_1}
M^{f,a}_{\cdot} \Rightarrow   \{ B^f_t, t\geq 0\} :=\{\tilde B_{Dt}, t\geq 0\},
\end{equation}
where $\tilde B_{Dt}$ is 
a Brownian motion with variance 
$D$. 
\end{proof}
 Define
\begin{align}
A_t&:= 
\int_0^t \int_{{\R}}w_s(x)(1-w_s(x))\,dx\,ds,\end{align}
and its rescaled version
\[
A^{a}_{t}= \frac{1}{a^2}A_{a^2t},~~a>0. 
\]
\begin{cor} We have convergence in law
\label{cor2}
\begin{equation}
(L^a,R^a, M^a, A^a,M^{f,a}, A^{f,a}) \Rightarrow 
\{(B_t,B_t, B_t,t,B^f_t, Dt),\; t\geq 0\},
\end{equation} 
in $\bD_{\R^5}$, as $a\rightarrow \infty.$  
Here, $B_t$ is a standard Brownian motion, $B_t^f$ is a Brownian motion  
with variance~$D$
and their correlation is given by 
\begin{equation}
\langle B_\cdot, B^f_\cdot\rangle_t = c_f t, \;t\geq 0,
\end{equation}
with $c_f$ as in (\ref{eq:def-c-sub-f}). 
\end{cor}
\begin{proof}
It only remains to check the correlation:
\begin{align*}
\langle M^{f,a}_\cdot, M^a_\cdot\rangle_t &= a^{-2}\int_0^{a^2t}\int_{\R} 
f(w_s(x)) \,dx\,ds 
\Rightarrow  t \E_{w,st}\Big[\int_{{\R}}
f(w(x))\,dx\Big]=  c_f t, \;t\geq 0,
\end{align*}
as $a\rightarrow\infty$, exactly as in~\eqref{6_2}. 
\end{proof}

\section{The proof of Theorem~\ref{th:2}: the upper bound on the speed}\label{sec:upper}

We assume till the end of the paper, without loss of generality, that $c_f>0$.
In this section, we prove the upper bound on the front speed in Theorem~\ref{th:2}. 
\begin{prop}
\label{prop:1}
Suppose that $u_0$ satisfies (\ref{feb1206}) 
and $f$ satisfies  (\ref{feb1202a}).
Then with probability 1,  we have 
\[
\limsup_{\sigma\to\infty}\sigma^2V(\sigma)\leq c_f\,.
\]
\end{prop}

\subsection{Rescaling}
\label{sec:4}

First, we show via a rescaling how to pass from (\ref{eq:RFKPP-general}) to (\ref{eq:spde-v_f}).
Consider the rescaled function 
\[
v_t(x)=u_{\sigma^{-4}t}(\sigma^{-2}x).
\]
To get an equation for $v_t(x)$, we use the mild form \eqref{eq:mild-form} and 
the relations 
\begin{align}
\label{eq:scaling-g-w}
G_{a^2t}(bx)&=b^{-1}G_{(a^2t/b^2)}(x)  \\
W^{a,b}(dyds)&:=a^{-1}b^{-1/2}W(bdy,a^2ds)\stackrel{\mathcal{D}}{=}W(dyds)  \nonumber\\
ab^{1/2}W^{a,b}(dyds)&=W(bdy,a^2ds),  \nonumber
\end{align}
that hold for any $a,b>0$. Here, $\stackrel{\mathcal{D}}{=}$ means equality in distribution.  
From \eqref{eq:mild-form}, for any $a,b>0$, we get 
\begin{align*}
u_{(a^2t)}(bx)&=\int_{{\R}}G_{a^2t}(bx-y)u_0(y)dy
+\int_{0}^{a^2t}\int_{{\R}}G_{a^2t-s}(bx-y)f(u_s(y))dyds  \\
&\quad +\sigma\int_{0}^{a^2t}\int_{{\R}}G_{a^2t-s}(bx-y)
 \sqrt{u_s(y)(1-u_s(y))}W(dyds)=: I+II+III.
\end{align*}
We make the change of variables $s=a^2s'$, $y=by'$ and use 
\eqref{eq:scaling-g-w}.  
For the term $I$ we have
\begin{align}
\label{eq:term-I}
I=b\int_{{\R}}G_{a^2t}(bx-by')u_0(by')dy' 
= \int_{{\R}}G_{a^2t/b^2}(x-y')u_0(by')dy'.   
\end{align}
The second term can be rewritten as
\begin{align}
\label{eq:term-II}
II&=\int_{0}^{a^2t}\int_{{\R}}G_{a^2t-s}(bx-y)f(u_s(y))dyds  
=\int_{0}^{t}\int_{{\R}}G_{a^2t-a^2s'}(bx-by')
f(u_{a^2s'}(by')) ba^2dy'ds'  \nonumber\\
&=a^2\int_{0}^{t}\int_{{\R}}G_{a^2(t-s')/b^2}(x-y')
 f(u_{a^2s'}(by')) dy'ds', 
\end{align}
and changing variables, the last term is
\begin{align}
\label{eq:term-III}
III&=\sigma\int_{0}^{a^2t}\int_{{\R}}G_{a^2t-s}(bx-y)
 \sqrt{u_s(y)(1-u_s(y))}W(dyds)   \\
&=
 \sigma\int_{0}^{t}\int_{{\R}}G_{a^2t-a^2s'}(bx-by')  
\sqrt{u_{a^2s'}(by')(1-u_{a^2s'}(by'))}W(bdy',a^2ds')  
 \nonumber\\
&= \sigma\int_{0}^{t}\int_{{\R}}b^{-1}G_{a^2(t-s')/b^2}(x-y')
\sqrt{u_{a^2s'}(by')(1-u_{a^2s'}(by'))}ab^{1/2}W^{a,b}(dy'ds') . 
 \nonumber\\
&= ab^{-1/2}\sigma\int_{0}^{t}\int_{{\R}}G_{a^2(t-s')/b^2}(x-y')
\sqrt{u_{a^2s'}(by')(1-u_{a^2s'}(by'))}W^{a,b}(dy'ds') . 
 \nonumber
\end{align}
We take
\[
a=\sigma^{-2}, \quad b=\sigma^{-2}, 
\]
so that $ab^{-1/2}\sigma=1$ and $a^2/b^2=1$. 
Defining $v_t(x):=u_{(a^2t)}(bx)$ and putting together the above terms, we see
that $v_t(x)$ satisfies 
\begin{align}
\label{eq:W-ab}
v_t(x)&=\int_{{\R}}G_{t}(x-y)u_0(y)dy
+\sigma^{-4}\int_{0}^{t}\int_{{\R}}G_{t-s}(x-y)
 f(v_s(y))dyds  \\
&\quad +\int_{0}^{t}\int_{{\R}}G_{t-s}(x-y)
 \sqrt{v_s(y)(1-v_s(y))}W^{a,b}(dyds).   \nonumber
\end{align}
Since the solution $v$ to \eqref{eq:W-ab} is unique in law, and since 
$W$ and $W^{a,b}$ are equal in law, we see that $v$ is the unique weak 
solution to \eqref{eq:spde-v_f} with the initial condition
$v_0(x) = u_0(\sigma ^{-2}x)$. 
Thus in general our scaling changes the initial data. However, if $u_0(x)=\1({x\le 0})$, then clearly $v_0(x)=u_0(x)$. 

Now it is clear that the conclusion of Proposition~\ref{prop:1} 
would follow 
if we show that
\begin{align}
\label{eq:eqiv-thm-1}
\limsup_{\sigma\to\infty}\sigma^4V^{(v)}(\sigma)&\leq c_{f},
\end{align}
where
\[
V^{(v)}(\sigma)=\lim_{t\to+\infty}\farc{R(v_t)}{t}.
\]
Let us also note that the rescaled Girsanov functional (\ref{eq:Girsanov-u}) 
takes the form 
\begin{align}
\label{eq:Girsanov}
Z_t&=\sigma^{-4}\int_{0}^{t}\int_{{\R}}
\frac{f(w_s(x))}{\sqrt{ w_s(x)(1-w_s(x))}}W(dx,ds)
  -\frac{1}{2}\sigma^{-8}\int_{0}^{t}\int_{{\R}}\frac{f(w_s(x))^2}{w_s((x)1-w_s(x))}dxds  \\
&=\sigma^{-4} M^f_t - \frac{1}{2}\sigma^{-8}A^f_t.    \nonumber
\end{align}

\subsection{Time steps for the upper bound}

 Note that for the upper bound on $V(\sigma)$ and $ V^{(v)}(\sigma)$, we may 
assume without loss of generality that the initial condition $u_0(x)=v_0(x)=\1({x\le 0})$, 
by the comparison principle and translation invariance in law.
We will define a sequence of stopping times $0=\tau_0\leq\tau_1\leq\cdots$, and
a sequence $v^{(m)}_t(x)$ of solutions to~(\ref{eq:spde-v_f}) for
$t\geq\tau_m$, with the initial conditions $v^{(m)}_{\tau_m}(x)\ge v_{\tau_m}(x)$ at $t=\tau_m$.
The comparison principle will imply that $v^{(m)}_t(x)\ge v_t(x)$ for 
$t\ge\tau_m$.  Moreover, we will choose $v_t^{(m)}$ so that for each 
$m=0,1,2,\ldots$ the following conditions hold almost surely:
\begin{align}
\label{eq:goal-v-comparison-a}
v_t(x)&\leq v_t^{(m)}(x), \quad  
 \text{for }t\geq\tau_m,\, x\in{\R}  \\
\label{eq:goal-v-comparison-b}
R\big(v^{(m)}_{\tau_m}\big)&= m\lambda_1\sigma^4,  \\
\label{eq:goal-v-comparison-c}
v_{\tau_{m}}^{(m-1)}(x)&\leq v_{\tau_{m}}^{(m)}(x),  
 \quad \text{for }x\in{\R}, 
\end{align}
with the constant $\lambda_1$ to be specified later.  
It follows from  (\ref{eq:goal-v-comparison-a}), 
that for all $m=0,1,2,\ldots$ and for all
$t\geq\tau_m$, we have
\begin{equation}
\label{eq:R-comparison}
R(v_t)\leq R\big(v^{(m)}_t\big),
\end{equation}
almost surely.
Thus, to bound $R(v_t)$ from above, it suffices to bound $R(v_t^{(m)})$.  

Let us inductively construct $\tau_k$ and $v_t^{(k)}(x)$ for $k=0,1,\ldots$.  
For convenience in \eqref{eq:goal-v-comparison-c}, we define 
$v^{(-1)}_{t}(x)=0$.  Fix $T_0>0$ and
$N\in{\N}$, to be specified later, and start with $\tau_0=0$ and  
$v_t^{(0)}(x)=v_t(x)$, so that \eqref{eq:goal-v-comparison-a}, 
\eqref{eq:goal-v-comparison-b}, and \eqref{eq:goal-v-comparison-c} hold 
for $m=0$ automatically.  Suppose that we have defined $\tau_m$ and  
$v^{(m)}_t$ for $t\geq\tau_m$ and $0\leq m\leq k$, and assume
that (\ref{eq:goal-v-comparison-a}), (\ref{eq:goal-v-comparison-b}) 
and (\ref{eq:goal-v-comparison-c})
hold for $0\leq m\leq k$.   Given $v^{(k)}_t(x)$, defined for 
$t\geq\tau_{k}$ and $x\in{\R}$, we set  
\begin{align}\label{feb1104}
M^{f,v,k}_t:&= \int_{\tau_{k}}^t \int_{{\R}}
\frac{f\big(v^{(k)}_s(x)\big)}{\sqrt{ v^{(k)}_s(x)(1-v^{(k)}_s(x))}}W(dx,ds),
\end{align}
and
\begin{align}\label{feb1102}
\tau_{k+1}&=\inf\Big\{t\in[\tau_k,\tau_k+T_0\sigma^8]:
  R\big(v_t^{(k)}\big)=(k+1)\lambda_1\sigma^4 
~~\text{or}~~\;\frac{1}{\sigma^4}M^{f,v,k}_t\geq N\Big\},
\end{align}
with $\tau_{k+1}=\tau_k+T_0\sigma^8$ if the above set is empty.
Then, we define $v^{(k+1)}_t(x)$ for $t\geq\tau_{k+1}$, and~$x\in{\R}$ as the 
solution to (\ref{eq:spde-v_f}) with the initial condition
\[
v^{(k+1)}_{\tau_{k+1}}(x)=\1(x\leq (k+1)\lambda_1\sigma^4).
\]
Note that for $m=k+1$, \eqref{eq:goal-v-comparison-a} and 
\eqref{eq:goal-v-comparison-c} hold by the comparison principle.  
\eqref{eq:goal-v-comparison-b} holds by construction.  

For convenience, we write
\[
\Delta\tau_m=\tau_{m+1}-\tau_m
\]
and note that $\{\Delta\tau_m\}$ are i.i.d. random variables for $m\ge 0$.  

\subsection{A good event and its consequences (for the upper bound)}

To get an upper bound on $V^{(v)}_f(\sigma)$, it suffices to get an appropriate 
lower bound on $\tau_m$ as $m\to\infty$.  
Let us define the event
\begin{equation}\label{feb1110}
G_{m}=\{\Delta\tau_m=T_0\sigma^8\},~~m\ge 0.
\end{equation}
Proposition~\ref{prop:1} is a consequence of the following lemma.
\begin{lemma}
\label{lem:good}
Let $\eps\in(0,\min(10^{-1},  c_f^{-2}))$ be arbitrary and set $\delta_\eps=\eps/10$. 
There exist $\bT_\eps$, $N_\eps$, and $\sigma_\eps$ such that 
for $T_0=\bT_\eps$, $N=N_\eps$
and any $\sigma\geq\sigma_0^\eps$, $m\geq 0$, and 
\begin{equation}
\label{9_1_11}
\lambda_1=(c_f+\delta_\eps)\bT_\eps,
\end{equation}
we have
\begin{equation}
\label{eq:bound-G-0}
\lambda_2:=\Pm_v\big(G_{m}\big)\geq 1-\delta_\eps. 
\end{equation}
\end{lemma}
Note that $\lambda_2$ does not depend on $m$ since $\Delta\tau_m$ are i.i.d for
$m\ge 1$.
We will prove Lemma \ref{lem:good} in the next section.  
Now we are ready to give
\begin{proof}[Proof of Proposition~\ref{prop:1}]
Given $\eps\in(0,1/10)$, let 
$\bT_\eps$, $N_\eps$ and $\sigma_\eps$ 
be as in Lemma~\ref{lem:good}, and take an arbitrary~$\sigma\ge\sigma_\eps$.
Then by Lemma~\ref{lem:good}, we have
\begin{equation}
\label{9_1_10}
\lambda_2 \geq 1-\delta_\eps, 
\end{equation} 
and
by \eqref{eq:bound-G-0} and the definition of $G^{(m)}$ with $T_0=\bT_\eps$, we get
\[
\E_v[\Delta\tau_m]\geq \bar T_\eps\lambda_2\sigma^8.
\]
The strong law of large numbers implies that we have, $\Pm_v$ 
almost surely,
\begin{equation}\label{taum-m}
\lim_{m\to\infty}\frac{\tau_m}{m}\geq \bT_\eps \lambda_2\sigma^8.
\end{equation}
Since $R(v^{(m)}_{\tau_m})=m\lambda_1\sigma^4$, we have that, also $\Pm_v$ almost 
surely, 
\begin{equation}
\label{eq:G-0-conclusion}
\limsup_{m\to\infty}\frac{R(v^{(m)}_{\tau_m})}{\tau_m}
=\limsup_{m\to\infty}\frac{m\lambda_1\sigma^4}{\tau_m}
\leq \frac{\lambda_1\sigma^4}{\bT_\eps\lambda_2\sigma^8}
= \frac{\lambda_1}{\bT_\eps\lambda_2}\sigma^{-4}.
\end{equation}
Furthermore, since by definition, for $\tau_m\leq t\leq\tau_{m+1}$ we have
\[
R\big(v^{(m)}_t\big)\leq  (m+1)\lambda_1\sigma^4. 
\]
Hence, we get that, $\Pm_v$ almost surely, we have,
using (\ref{9_1_10}) and (\ref{taum-m}),
\begin{align}
\label{eq:to-verify-for-velocity}
V^{(v)}(\sigma)
&\le\limsup_{m\to\infty}\sup_{\tau_m\leq t\leq\tau_{m+1}}\le
\frac{R\big(v^{(m)}_{t}\big)}{t}\le 
 \limsup_{m\to\infty}\frac{\lambda_1(m+1)\sigma^4}{\tau_m}
 \\
\nonumber
&\le \frac{\lambda_1\sigma^4}{\lambda_2\bT_\eps\sigma^8}
\leq \frac{(c_f+\delta_\eps)}{(1-\delta_\eps)}\sigma^{-4}
\leq (c_f+\sqrt\eps)\sigma^{-4}. 
\end{align}
Note that \eqref{eq:to-verify-for-velocity} holds for any $\sigma\geq \sigma_0^\eps$, 
and, since $\eps$ is arbitrary small, we are done.   
This finishes the proof of Proposition \ref{prop:1}. 

\end{proof}

\subsection{Proof of Lemma \ref{lem:good}}

As $G_{m}$ are i.i.d., it suffices to set $m=0$.
We fix $\eps\in (0,1/10)$, let $\delta_\eps=\eps/10$,  
take $\bT_\eps$ sufficiently large, so that 
\begin{align} 
\label{9_12_2}
2 \exp\Big(-\frac{\delta_\eps^2\bT_\eps}{2}\Big)
 &\leq \farc{\eps}{100},
\end{align} 
set $\lambda_1=(c_f+\delta_\eps)\bT_\eps$, and 
let $N_\eps> (2+\delta_\eps)\bT_\eps D$  
be sufficiently large (its value will be determined later in the proof).
We define the stopping time 
\begin{equation}
\label{10_1_3}
\xi^\eps
=\inf\{t\geq 0: M^{f,\sigma^4}_t \geq N_\eps\}.  
\end{equation}
Then by Girsanov's theorem, we have, with $Z_t$ as in ~\eqref{eq:Girsanov}:
\begin{align*}
\nonumber
\Pm_v&(G^c_0)=\E_w\left[\exp\left(Z_{\sigma^8(\bT_\eps\wedge \xi^\eps)}\right)\1_{G^c_0}\right] 
= \E_w\Big[\exp\left(\sigma^{-4}\Big(M^f_{\sigma^8(\bT_\eps\wedge \xi^{\eps})}-
\frac{1}{2}\sigma^{-4}  A^f_{\sigma^8(\bT_\eps\wedge \xi^{\eps})}\Big)\right)  \\
&\hspace{1.5cm} \times\1\big(R(w_t)\geq 
\lambda_1 \sigma^4 \;\text{for some}\; t\leq \sigma^8\bT_\eps\; \text{or}\;  
\sigma^{-4}M^{f}_t \geq N_\eps   \;\text{for some}\; t\leq \sigma^8 \bT_\eps
\big)\Big]   \\
&= \E_w\left[\exp\left(M^{f,\sigma^4}_{\bT_\eps\wedge \xi^{\eps}}
- \frac{1}{2}A^{f,\sigma^4}_{\bT_\eps\wedge \xi^{\eps}}\right)
 \times\1\left(R^{\sigma^4}_t\geq (c_f+\delta_\eps)\bT_\eps  
 \;\text{for some}\; t\leq \bT_\eps
 \;  \text{or}\;  
\xi^{\eps} \leq \bT_\eps   
 \right)\right]   \\
&\leq \E_w\Big[\exp\left(M^{f,\sigma^4}_{\bT_\eps\wedge \xi^{\eps}}- 
\frac{1}{2}A^{f,\sigma^4}_{\bT_\eps\wedge \xi^{\eps}}\right)
\1\left(\xi^{\eps}\leq \bT_\eps\right)\Big]   \\
&\hspace{0.5cm}+ \E_w\Big[\exp\left(M^{\sigma^4}_{\bT_\eps\wedge \xi^{N,\sigma}}
  -\frac{1}{2}A^{f,\sigma^4}_{\bT_\eps\wedge \xi^{N,\sigma}}\right)   
  \times\1\left(R^{\sigma^4}_t\geq (c_f+\delta_\eps)\bT_\eps  
  \;\text{for some}\; t\leq 
     \bT_\eps\right)\1\left(\xi^{\eps}> T_\eps\right)\Big]   \\
&=: I^{\eps}_1+I^{\eps}_2\,.
\end{align*}
We first bound $ I^{\eps}_1$: 
\begin{align}
\label{9_1_5}
I^{\eps}_1
&= \E_w\left[\exp\left(M^{f,\sigma^4}_{\xi^{\eps}\wedge \bT_\eps}- 
\frac{1}{2}A^{f,\sigma^4}_{\xi^{\eps}\wedge \bT_\eps}\right) 
\1\left(\xi^{\eps}\leq \bT_\eps\right)\right] 
\leq e^{N_\eps} \Pm_w\Big(\sup_{0\le t\leq \bT_\eps} M^{f,\sigma^4}_t \geq N_\eps\Big).
\end{align} 
Let $\Pm^B$ and $\Pm^{B^f}$ be the measures induced on the canonical path space 
by  the standard Brownian motion $B$ and by the Brownian motion with variance $D$,
respectively, and $\E^B$ and $\E^{B^f}$ be the corresponding expectations.
Then by Lemma~\ref{lem6_1} we have 
\begin{align}
\limsup_{\sigma\rightarrow \infty} I^{\eps}_1
\leq  e^{N_\eps} \Pm^{B^f}\Big(\sup_{0\le t\leq \bT_\eps} B^f_t \geq N_\eps\Big)
\leq e^{N_\eps}2  \Pm^{B}\left( \sqrt{D} B_{\bT_\eps} \geq N_\eps\right) 
\label{9_12_1}
\leq 2e^{N_\eps}e^{-{N_\eps^2}/(2\bT_\eps D)},
\end{align}
where the second inequality follows by the reflection principle and the last inequality 
follows by a simple bound on Gaussian tail probabilities. 
By choosing  $N_\eps$ 
sufficiently large, 
we get 
\begin{align}
\nonumber
\limsup_{\sigma\rightarrow \infty} I^{\eps}_1&\leq  \eps/100. 
\end{align}
Thus, 
there exists $\sigma_\eps$, such that for all $\sigma\geq \sigma_\eps$ we have 
\begin{align}
\label{9_1_6}
 I^{\eps}_1&\leq  \eps/50. 
\end{align}

Next, we bound $I^{\eps}_2$. 
Let $\Pm^{B^f,B}$  be the measure induced on the canonical path space by  
the zero-mean Brownian motions $B^f,B$, 
such that $B^f$ has variance $D$, $B$ has variance $1$, and the 
covariance of $B^f$ and $B$ is $c_f$, 
and let $\E^{B^f,B}$ be the corresponding expectation.
We use again  Corollary~\ref{cor2}, 
properties of weak convergence, the dominated convergence theorem (we can 
switch to the Skorohod space if needed) to get  
\begin{align}
\nonumber
&\limsup_{\sigma\rightarrow \infty} I^{\eps}_2  
= \limsup_{\sigma\rightarrow \infty} 
\E_w\Big[\exp\big(M^{f,\sigma^4}_{\bT_\eps}- \frac{1}{2} A^{f,\sigma^4}_{\bT_\eps}\big)
\1(\sup_{0\le t\leq \bT_\eps} R^{\sigma^4}_t\geq (c_f+\delta_\eps)\bT_\eps)  
\nonumber
\1(\sup_{0\le t\leq \bT_\eps} M^{f,\sigma^4}_t<N_\eps)\Big]   \nonumber\\
&~~~~\leq \E^{B^f,B}\Big[e^{B^f_{\bT_\eps}- \frac{1}{2}D\bT_\eps)}
 \1\Big(\sup_{0\le t\leq \bT_\eps} B_t\geq (c_f+\delta_\eps)\bT_\eps\Big)
 \1\Big(\sup_{0\le t\leq \bT_\eps} B^f_t\leq N_\eps\Big)\Big] 
  \\
&~~~~\leq \E^{B^f,B}\Big[e^{B^f_{\bT_\eps}- \frac{1}{2}D\bT_\eps}
 \1(\sup_{0\le t\leq \bT_\eps} B_t\geq (c_f+\delta_\eps)\bT_\eps)\Big] 
\nonumber
= \Pm^B\Big(\sup_{0\le t\leq \bT_\eps}(B_t+c_ft)  \geq  (c_f+\delta_\eps)\bT_\eps\Big).
\end{align} 
In the last equality we used the Girsanov theorem, since under the
$\exp(B^f_{\bT_\eps}- \frac{1}{2}D\bT_\eps)$ change of measure,  
$B$ is a Brownian motion with the
drift $2c_f$ (recall that the covariance of $B^f$ and $B$ is $c_f$).  
Now it is easy to get 
\begin{align} 
\Pm^B\Big(\sup_{0\le t\leq \bT_\eps}(B_t+c_ft)  \geq  (c_f+\delta_\eps)\bT_\eps\Big) 
\leq \Pm^B\Big(\sup_{0\le t\leq \bT_\eps}B_t  \geq  \delta_\eps \bT_\eps\Big)
\leq 2 e^{-{\delta_\eps^2\bT_\eps}/{2}}
 &\leq \eps/100,
\end{align} 
where in the second inequality we again used reflection principle and 
a bound on Gaussian tail, and the last inequality follows from
\eqref{9_12_2}. Hence, there is $\sigma_\eps$ 
such that for all~$\sigma\geq \sigma_\eps$, we have 
\begin{align}
\nonumber
 I^{\eps}_2&\leq  \eps/50. 
\end{align}
Combining the above estimates, we get that for $\sigma\geq \sigma_\eps$ we have 
\begin{align}
\Pm_v(G^c_0)&\leq 2\eps/50\leq \eps/10,
 \label{eq:ubound_G3}
\end{align}
so that
\begin{equation} 
\Pm_v(G_0)\geq 1-\eps/10. 
\end{equation}
This finishes the proof of Lemma \ref{lem:good}.~$\qed$

\section{Proof of Theorem~\ref{th:2}: the lower bound on the speed}\label{sec:lower}

We now prove the lower bound on $V(\sigma)$. 
\begin{prop}
\label{prop:2}
Suppose that $u_0$ satisfies (\ref{feb1206})
and $f$ satisfies (\ref{feb1202a}).  Then with 
probability 1,  we have  
\[
\liminf_{\sigma\to\infty}\sigma^2V(\sigma)\geq c_f\,.
\]
\end{prop}
The proof of Proposition~\ref{prop:2} follows a similar strategy to that
of Proposition~\ref{prop:1}. As in the proof of the upper bound,
using the comparison principle and shift invariance in law, 
we may assume without loss
of generality that $u_0(x)=v_0(x)=\1(x\le 0)$. 

\subsection{Time steps for the lower bound}
\label{subsec:3_1}

We start with the  definition of  the time steps. The main difference with 
the proof of the upper bound is that we will sometimes update ``backwards'',
 and that the "good events" will be when the stopping time happens before
a fixed time length rather than when the stopping times happen at a deterministic
time steps, as in (\ref{feb1110}) in the proof of the upper bound. 
We will define stopping times~$0=\tau_0\leq\tau_1\leq\cdots$,
and a sequence $v^{(m)}_t(x)$ of random processes, which 
will be solutions to~(\ref{eq:spde-v_f}), for
$t\geq\tau_m$, 
such that for each 
$m=0,1,2,\ldots$ the following conditions will hold almost surely:
\begin{align}
\label{eq:goal-v-comparison-a1}
v_t(x)&\geq v_t^{(m)}(x), \qquad  
 ~~\text{ for }t\geq\tau_m,\, x\in{\R}  \\
\label{eq:goal-v-comparison-c1}
v_{\tau_{m}}^{(m-1)}(x)&\geq v_{\tau_{m}}^{(m)}(x),
 \qquad \text{for }x\in{\R}. 
\end{align}
Given (\ref{eq:goal-v-comparison-a1}) and (\ref{eq:goal-v-comparison-c1}), 
it would follow almost surely for all $m=0,1,2,\ldots$ and for all
$t\geq\tau_m$, that
\begin{equation}
\label{eq:R-comparison1}
L(v_t)\geq L(v^{(m)}_t).
\end{equation}
Thus, to bound $L(v_t)$ from below, it would suffice to bound $L(v_t^{(m)})$.  

We now describe the induction, starting with $\tau_0=0$, and   
$v_t^{(0)}(x)=v_t(x)$, so that  \eqref{eq:goal-v-comparison-a1} 
holds for~$m=0$.  Also define $v^{(-1)}_t(x)=1$, so that 
\eqref{eq:goal-v-comparison-c1} holds.  
Let us fix some constants $\tl_1, \wT_0, N>0$, to be specified later.  
Suppose that we have defined $\tau_m$ for $0\leq m\leq k$ and  
$v^{(m)}_t$ for $t\geq\tau_m$ and $0\leq m\leq k$, and  
that (\ref{eq:goal-v-comparison-a1}) and (\ref{eq:goal-v-comparison-c1}) 
hold for $0\leq m\leq k$.
To define $\tau_{k+1}$, we consider, as in (\ref{feb1104}), 
\begin{align}\label{feb1108}
M^{f,v,k}_t:&= \int_{\tau_k}^t 
\int_{{\R}}\frac{f(v^{(k)}_s(x))}{\sqrt{ v^{(k)}_s(x)(1-v^{(k)}_s(x))}}W(dx,ds),
\end{align}
and set
\begin{align}\label{feb1106}
\tau_{k+1}&=\inf\Big\{t\in[\tau_k,\tau_k+\wT_0\sigma^8]:
  |L\big(v_t^{(k)}\big)-L\big(v_{\tau_k}^{(k)}\big)| \geq \lambda_1\sigma^4 \;
\text{ or }\;\frac{1}{\sigma^4}M^{f,v,k}_t\geq N\Big\}
\end{align}
with the convention $\tau_{k+1}=\tau_k+\wT_0\sigma^8$ if the above set is empty. 

We then let $v^{(k+1)}_t(x)$ for $t\geq\tau_{k+1}$, $x\in{\R}$ be the 
solution to (\ref{eq:spde-v_f}) with  the initial condition
\[
v^{(k+1)}_{\tau_{k+1}}(x)=
\left\{
\begin{array}{lcr}
\1(x\leq  L(v^{(k)}_{\tau_{k+1}})),\;&& \text{if}\; 
\tau_{k+1}<\tau_{k}+\wT_0\sigma^8,\\
\1(x\leq  L(v^{(k)}_{\tau_{k}})-\lambda_1\sigma^4),\;&& 
\text{if}\; \tau_{k+1}=\tau_{k}+\wT_0\sigma^8.
\end{array}\right.
\]
Then for $m=k+1$, the comparison principle gives 
\eqref{eq:goal-v-comparison-a1}, and \eqref{eq:goal-v-comparison-c1} is 
true by definition.  

As before, we write
\[
\Delta\tau_k=\tau_{k+1}-\tau_k
\]
and
\begin{equation}
\Delta L_k = L\big(v_{\tau_{k+1}}^{(k)}\big)-L\big(v_{\tau_k}^{(k)}\big) 
\end{equation}
Note that $\{(\Delta\tau_m, \Delta L_k)\}$ are i.i.d. random variables.

\subsection{A good event and its consequences for lower bound}
We define the "good" events
\[
\widetilde G_0^{(m)}=\{\Delta\tau_m < \wT_0\sigma^8\}.
\]
and 
\begin{align*}
\widetilde G_0^{(1,m)}&=\Big\{\Delta\tau_m < \wT_0\sigma^8,~ 
\Delta L_k=\tl_1 \sigma^4, 
\sup_{\tau_m\le t\le \tau_{m+1}}\frac{1}{\sigma^4}M^{f,v,k}_t < N \Big \},\\
\widetilde G_0^{(2,m)}&=\Big\{\Delta\tau_m < \wT_0\sigma^8, ~
\Delta L_k=-\tl_1  \sigma^4, 
\sup_{\tau_m\le t\le \tau_{m+1}}\frac{1}{\sigma^4}M^{f,v,k}_t < N\Big\},\\
\widetilde G_0^{(3,m)}&=\Big\{\Delta\tau_m < \wT_0\sigma^8, 
\sup_{\tau_m\le t\le \tau_{m+1}}\frac{1}{\sigma^4}M^{f,v,k}_t = N\Big\}.
\end{align*}
To get  a lower  bound on 
$V^{(v)}(\sigma)$ we need a lower bound on $\Delta L_m$ as $m\to\infty$. 
To this end the following lemma will be helpful. 
\begin{lemma}
\label{lem:13_1}
Let $\eps\in(0,1/10)$ be arbitrary and $\delta_\eps=\eps/10$. 
There exist $T^\ast_\eps$, $N_\eps$ and $\sigma_\eps$
so that for all~$\sigma\geq\sigma_\eps$,  $m\geq 0$, $\wT_0=T^\ast_\eps, N=N_\eps$   and 
\begin{equation}
\label{9_1_11a}
\tl_1=(c_f-\delta_\eps)T_\eps^\ast, 
\end{equation}
we have
\begin{align}
\label{eq:bound-G-0_a}
\Pm_v\big(\widetilde  G_0^{(1,m)}\big)&\geq 1-\eps/50,\\
\label{eq:bound-G-0_a1}
\Pm_v\big(\widetilde  G_0^{(2,m)}\big)&\leq \eps/20,\\
\label{eq:bound-G-0_a3}
\Pm_v\big(\widetilde  G_0^{(3,m)}\big)&\leq \eps/50. 
\end{align}
\end{lemma}
We postpone the proof of this lemma and first give
\begin{proof}[Proof of Proposition~\ref{prop:2}]
Let us take $\eps\in(0,\min(10^{-1}, c^{-2}_f))$, 
and 
choose $T_\eps^\ast$, $N_\eps$ and $\sigma_\eps$ 
as in Lemma~\ref{lem:13_1}, and  consider an arbitrary $\sigma\ge\sigma_\eps$. 
Lemma~\ref{lem:13_1} implies that 
\begin{equation}
\label{16_12_6}
\Pm_v(\widetilde  G_0^{(1,m)}) -\Pm_v(\widetilde  G_0^{(2,m)})
-\Pm_v(\widetilde  G_0^{(3,m)}) - \big(1-\Pm_v(\widetilde  G_0^{(m)})\big)\geq 
1-\eps/5, 
\end{equation}
for all $\sigma\ge \sigma_\eps$, so that for all $m\ge 0$ we have 
\begin{align}
\nonumber
\E_v[\Delta L_m]&\geq \tl_1\sigma^4 \Big(\Pm_v(\widetilde  G_0^{(1,m)})- 
\Pm_v(\widetilde  G_0^{(2,m)})- P_v(\widetilde  G_0^{(3,m)})- 
\big(1-P_v(\widetilde G_0^{(m)})\big)\Big)\\
&\geq \tl_1	\sigma^4 (1-\eps/5). 
\end{align}
Then using the strong law of large numbers, we have that $\Pm_v$ 
almost surely,
\[
\lim_{m\to\infty}\frac{L(v^{(m)}_{\tau_m})}{m}\geq \tl_1 (1-\eps/5)\sigma^4,
\]
for all $\sigma\ge \sigma_\eps$.  
Since $\tau_m\leq m T_\eps^\ast\sigma^8$, we have that $\Pm_v$ almost 
surely, 
\begin{equation}
\label{eq:G-0-conclusion_1}
\liminf_{m\to\infty}\frac{L(v^{(m)}_{\tau_m})}{\tau_m}
=\liminf_{m\to\infty}\frac{L(v^{(m)}_{\tau_m})}{m}\frac{m}{\tau_m}
\geq \frac{\tl_1(1-\eps/5)\sigma^4}{T_\eps^\ast\sigma^8}
= \frac{\tl_1(1-\eps/5)}{T_\eps^\ast}\sigma^{-4}.
\end{equation}
Furthermore, since  for $\tau_m\leq t\leq\tau_{m+1}$ we have
\begin{align*}
L\big(v^{(m)}_t\big)&\geq L\big(v^{(m)}_{\tau_m}\big) -\tl_1\sigma^4,
\end{align*}
and $\Delta\tau_m\leq T_\eps^\ast\sigma^8$, 
it follows that, $P_v$ almost surely and since $\varepsilon<10^{-1}$, 
\begin{align}
\label{eq:to-verify-for-velocity_1}
V^{(v)}(\sigma)&\geq \liminf_{m\to\infty}\inf_{\tau_m\leq t\leq\tau_{m+1}}
\frac{L\big(v^{(m)}_{t}\big)}{t}
\geq \liminf_{m\to\infty}
\frac{L\big(v^{(m)}_{\tau_m}\big)-\tl_1\sigma^4}{\tau_m+T_\eps^\ast\sigma^8}
\nonumber
\geq  \frac{\tl_1(1-\eps/5)}{T_\eps^\ast}\sigma^{-4}
\\
&=  \frac{(c_f-\eps/10)T_\eps^*(1-\eps/5)}{T_\eps^\ast}\sigma^{-4}
\nonumber
\geq \big(c_f-\frac{\eps}{5}(c_f+1)\big)\sigma^{-4}\\
\nonumber
&\geq   (c_f-\sqrt\eps)\sigma^{-4}. 
\end{align}
In the second inequality above we used~\eqref{eq:G-0-conclusion_1} 
and the fact that \mbox{$\tau_m\rightarrow\infty$}, $\Pm_v$-a.s. since $\Delta \tau_m\geq 0$, 
not identically zero and i.i.d. 
Since $\eps$ was chosen to be arbitrary small we are done.  
\end{proof}

\subsection{Proof of Lemma~\ref{lem:13_1}}

As
$(\Delta_{\tau_m}, \Delta L_m)$ are i.i.d., the events $G_0^{(m)}$ are also
i.i.d., hence we only need to prove~\eqref{eq:bound-G-0_a}-\eqref{eq:bound-G-0_a3} 
for~$m=0$ and write
\[
\widetilde  G_0=\widetilde  G_0^{(0)}, \widetilde  G^{(i)} _0= \widetilde G_0^{(i, 0)}, i=1,2,3. 
\]
Fix $\eps\in (0,10^{-1})$, let  $\delta_\eps=\eps/10$, and 
let $T_\eps^\ast$ be sufficiently large so that 
\begin{align} 
\label{8_1_2}
\Pm^B\Big(B_{1}\geq  -\delta_\eps \sqrt{T},~ \inf_{0\le t\leq 1} B_t>-(c_f-\delta_\eps)\sqrt{T}\Big)
\geq 1-\eps/100,\;\; \forall T\geq T_\eps^\ast\,.
\end{align} 
We consider $N_\eps> (2+\delta_\eps)T_\eps^\ast D$ sufficiently large, with a precise
value to be specified later, and define the stopping time 
\[
\xi^{\eps}=\inf\{t\geq 0: M^{f,\sigma^4}_t \geq N_\eps\}.  
\]
Then by Girsanov's theorem, 
and since
\[
\Big\{ \sup_{0\le t\leq\Delta\tau_0}\frac{1}{\sigma^4}M^{f}_t < N_\eps \Big\}  
\supset \Big\{\sup_{0\le t\leq T_\eps^\ast\sigma^8}\frac{1}{\sigma^4}M^{f}_t < N_\eps\Big\}
= \{\xi^{\eps}>T_\eps^\ast\},
\]
we have 
\begin{align}
\nonumber
\Pm_v&\big(\wG^{(1)}_0\big)=\E_w\left[\exp\left(
Z_{\sigma^8(T_\eps^\ast\wedge \xi^{\eps})}\right)\1_{\wG^{(1)}_0}\right]    
\geq E_w\left[\exp\left(\sigma^{-4}(M^f_{\sigma^8(T_\eps^\ast\wedge 
\xi^{\eps})}-\frac{1}{2}\sigma^{-4} 
  A^f_{\sigma^8(T_\eps^\ast\wedge \xi^{\eps})})\right) \right. \\
\nonumber&\hspace{1.5cm} \times\1\left(L(w(t))\geq \tl_1 \sigma^4 
\;\text{for some}\; 0\le t\leq \sigma^8(T_\eps^\ast\wedge \xi^{\eps})\right)\\
\nonumber&\hspace{1.5cm}  \times\1\left(L(w(t)) > -\tl_1 \sigma^4 \;\text{for all}\; 
0\le t\leq \sigma^8(T_\eps^\ast\wedge \xi^{\eps})\right)\\
\label{feb1116}
&\hspace{1.5cm} 
\left. 
\times\1\left(
\frac{1}{\sigma^4}M^{f}_t < N_\eps   \;\text{for all}\; 0\le t\leq 
\sigma^8 (T_\eps^\ast\wedge \xi^{\eps})
\right)\right]   \\
\nonumber&\hspace{1.1cm}
\geq \E_w\left[\exp\left(M^{f,\sigma^4}_{T_\eps^\ast}- \frac{1}{2} A^{f,\sigma^4}_{T_\eps^\ast}\right)
\1\left(L^{\sigma^4}_t>\tl_1  \;\text{for some}\; 0\le t\leq T_\eps^\ast
 \right)\right.  \\
\nonumber&\hspace{1.5cm}
 \times\1\left(L^{\sigma^4}_t>-\tl_1  \;\text{for all}\; 0\le t\leq T_\eps^\ast
 \right)  
\times\1\Big(\sup_{0\leq t\leq T_\eps^\ast} M^{f,\sigma^4}<N_\eps \Big)\Big] 
=: J^{\eps}_1\,.
\end{align}
Next, passing to the limit $\sigma\to+\infty$, we obtain, using the weak
convergence in Corollary~\ref{cor2}:
\begin{align}
\nonumber
\liminf _{\sigma\rightarrow \infty}& J^{\eps}_1
= \liminf _{\sigma\rightarrow \infty} \E_w\left[\exp\left(
M^{f,\sigma^4}_{T_\eps^\ast}- \frac{1}{2} A^{f,\sigma^4}_{T_\eps^\ast}\right)\right.
\times \1\left(L^{\sigma^4}_t> (c_f-\delta_\eps)T_\eps^\ast \;
\text{for some}\; 0\le t\leq T_\eps^\ast
 \right)\nonumber \\
\nonumber 
&\hspace{2.5cm}
 \times\1\left(L^{\sigma^4}_t>-(c_f-\delta_\eps)T_\eps^*  
 \;\text{for all}\; 0\le t\leq T_\eps^\ast
 \right)  
 \times\1\Big(\sup_{0\le t\leq T_\eps^\ast} 
 M^{f,\sigma^4}<N_\eps \Big)\Big]
\\
\nonumber
&\geq \E^{B^f,B}\Big[e^{B^f_{T_\eps^\ast}- \frac{1}{2}DT_\eps^\ast}
 \1\Big(\sup_{0\le t\leq T_\eps^\ast}
 B_t>(c_f-\delta_\eps)T_\eps^\ast\Big) 
\1\Big(\inf_{0\le t\leq T_\eps^\ast} B_t>-(c_f-\delta_\eps)T_\eps^\ast\Big) 
\\
 \label{9_1_8}
&\hspace{1.5cm}
\times 
\1\Big(\sup_{0\le t\leq T_\eps^\ast} B^f_t< N_\eps\Big)\Big] .
\end{align}
We rewrite this, using Girsanov's theorem for correlated Brownian motions with a drift, as
\begin{align}
  \liminf _{\sigma\rightarrow \infty} J^{\eps}_1
&\geq \E^{B^f,B}\Big[e^{B^f_{T_\eps^\ast}- \frac{1}{2}DT_\eps^\ast}
\1\Big(\sup_{0\le t\leq T_\eps^\ast} B_t>(c_f-\delta_\eps)T_\eps^\ast\Big)
 \times\1\Big(\inf_{0\le t\leq T_\eps^\ast} B_t>-(c_f-\delta_\eps)T_\eps^\ast\Big) \Big] 
  \nonumber\\
\nonumber &\quad - \E^{B^f,B}\Big[e^{B^f_{T_\eps^\ast}-  \frac{1}{2}DT_\eps^\ast}
\1\Big(\sup_{0\le t\leq T_\eps^\ast} B^f_t \geq N_\eps\Big)\Big] 
  \\
\nonumber
&= \Pm^B\Big(\sup_{0\le t\leq T_\eps^\ast}(B_t+c_ft)  \geq  (c_f-\delta_\eps)T_\eps^\ast, 
\inf_{0\le t\leq T_\eps^\ast} (B_t+c_ft)>-(c_f-\delta_\eps)T_\eps^\ast   \Big)
\\
\label{feb1114}
&\quad - \Pm^{B^f}\Big(\sup_{0\le t\leq T_\eps^\ast} B^f_t +Dt \geq N_\eps\Big).
\end{align} 
The first term in the right side can be bounded as
\begin{align}
\nonumber
\Pm^B&\Big(\sup_{0\le t\leq T_\eps^\ast}(B_t+c_ft)  
\geq  (c_f-\delta_\eps)T_\eps^\ast, \inf_{0\le t\leq T_\eps^\ast} (B_t+c_ft)>
-(c_f-\delta_\eps)T_\eps^\ast\Big)\\
\nonumber
&\geq \Pm^B\Big(B_{T_\eps^\ast}+c_fT_\eps^\ast  
\geq  (c_f-\delta_\eps)T_\eps^\ast, \inf_{0\le t\leq T_\eps^\ast} 
B_t 
>-(c_f-\delta_\eps)T_\eps^\ast\Big)
\\
\nonumber
&= \Pm^B\Big(B_{T_\eps^\ast}\geq  -\delta_\eps T_\eps^\ast, 
\inf_{0\le t\leq T_\eps^\ast} B_t>-(c_f-\delta_\eps)T_\eps^\ast\Big)
\\
\label{9_1_3}&
= \Pm^B\Big(B_{1}\geq  -\delta_\eps \sqrt{T_\eps^\ast},
\inf_{0\le t\leq 1} B_t>-(c_f-\delta_\eps)\sqrt{T_\eps^\ast}\Big)
\geq 1- \eps/100,
\end{align}
where the last inequality follows by~\eqref{8_1_2}. 
The second term in the right side of (\ref{feb1114}) 
can be bounded  using the reflection principle for Brownian motion, 
bounds on tails of Gaussian probabilities and by choosing  
$N_\eps\geq (2+\delta_\eps)T_\eps^\ast D$ sufficiently large, so that
\begin{align}
\nonumber
\Pm^{B^f}\Big(&\sup_{0\le t\leq T_\eps^\ast} B^f_t +Dt 
\geq N_\eps\Big)\leq \Pm^{B^f}\Big(\sup_{0\le t\leq T_\eps^\ast} 
B^f_t  \geq N_\eps-DT_\eps^\ast\Big)\\
& \label{9_1_2}
\leq 2\Pm^B\Big( \sqrt{D} B_{T_\eps^\ast} \geq N_\eps-DT_\eps^\ast\Big) 
\leq 2 
\exp\Big(-\frac{(N_\eps-DT_\eps^\ast)^2}{2T_\eps^\ast D}
\Big) 
\leq \eps/100.
\end{align}
Combining~\eqref{feb1116}-\eqref{9_1_2} we get that for $N_\eps$ sufficiently large
we have 
\begin{align}
\label{9_1_4}
\Pm_v&\big(\wG^{(1)}_0\big)\geq 1-\eps/50,
\end{align} 
which is \eqref{eq:bound-G-0_a}.

Next, we bound $\Pm_v(\wG_0^{(2)})$.
Again, using Girsanov's theorem we write 
\begin{align*}
\nonumber
\Pm_v&(\wG^{(2)}_0)=\E_w\left[\exp\left(
Z_{\sigma^8(T_\eps^\ast\wedge \xi^{\eps})}\right)\1_{\wG^{(2)}_0}\right]
\leq \E_w\Big[\exp\left(\sigma^{-4}(M^f_{\sigma^8(T_\eps^\ast\wedge \xi^{\eps})}
-\frac{1}{2}\sigma^{-4} 
A^f_{\sigma^8(T_\eps^\ast\wedge \xi^{\eps})}\right) \\
&\times\1
\left(L(w_t) \leq -\tl_1 \sigma^4 \;\text{for some}\; 
0\le t\leq \sigma^8(T_\eps^\ast\wedge \xi^{\eps})\right)\\
&
\times\1\Big(
\frac{1}{\sigma^4}M^{f}_t < N_\eps   \;\text{for all}\; 0\le t\leq \sigma^8T_\eps^\ast
\Big)\Big]   \\
\nonumber 
&
+ \E_w\Big[\exp\left(
\sigma^{-4}(M^f_{\sigma^8(T_\eps^\ast\wedge \xi^{\eps})}-\frac{1}{2}\sigma^{-4} 
  A^f_{\sigma^8(T_\eps^\ast\wedge \xi^{\eps})})\right)  
\times\1\Big(\sup_{0\le t\leq \sigma^8T_\eps^\ast}
\frac{1}{\sigma^4}M^{f}_t \geq  N_\eps 
\Big)   \Big]\\
&\hspace{-0.3cm}
\leq \E_w\Big[\exp\left(M^{f,\sigma^4}_{T_\eps^\ast}- \frac{1}{2}A^{f,\sigma^4}_{T_\eps^\ast}\right)
 \times\1\Big(\inf_{0\le t\leq T_\eps^\ast} L^{\sigma^4}_t\leq -\tl_1 
 \Big) 
 \times\1\Big(\sup_{0\le t\leq T_\eps^\ast} M^{f,\sigma^4}<N_\eps \Big)\Big] \\
 \nonumber 
&
+ \E_w\Big[\exp\Big(\sigma^{-4}
(M^f_{\sigma^8(T_\eps^\ast\wedge \xi^{\eps})}-\frac{1}{2}\sigma^{-4} 
  A^f_{\sigma^8(T_\eps^\ast\wedge \xi^{\eps})}\Big)  
\times\1\Big(\sup_{0\le t\leq \sigma^8T_\eps^\ast}
\frac{1}{\sigma^4}M^{f}_t \geq  N_\eps 
\Big)   \Big]
=: J^{\eps}_{2,1}+ J^{\eps}_{2,2}\,.
\end{align*}
The term $J_{2,2}^\eps$ is exactly as $I_1^\eps$ in~\eqref{9_1_5}, 
thus, as in~\eqref{9_1_6} we have, by choosing $N_\eps$ sufficiently large:
\begin{align}
\label{9_1_7}
 J^{\eps}_{2,2}&\leq  \eps/50, 
\end{align}
for all $\sigma$ sufficiently large. As for $J^{\eps}_{2,1}$,  
proceeding similarly to~\eqref{9_1_8},   we obtain
\begin{align}
\nonumber
\limsup _{\sigma\rightarrow \infty} J^{\eps}_{2,1} 
&\leq \E^{B^f,B}\Big[e^{B^f_{T_\eps^\ast}- \frac{1}{2}DT_\eps^\ast}
\1\Big(\inf_{0\le t\leq T_\eps^\ast} B_t\leq -(c_f-\delta_\eps)T_\eps^\ast\Big)  
\times \1\Big(\sup_{0\le t\leq T_\eps^\ast} B^f_t\leq  N_\eps\Big)\Big] 
  \nonumber\\
\nonumber
&\leq \Pm^B\Big( \inf_{0\le t\leq T_\eps^*} (B_t+c_ft)\leq 
-(c_f-\delta_\eps)T_\eps^\ast\Big)
\leq \Pm^B\Big( \inf_{0\le t\leq 1} B_t\leq -(c_f-\delta_\eps)\sqrt{T_\eps^\ast}\Big)\\
\label{9_1_9}
&\leq \eps/100.
\end{align}
Here, the last inequality follows from \eqref{8_1_2}. 
Combining~\eqref{9_1_7} and \eqref{9_1_9} we see 
that for $N_\eps$ sufficiently large we have
\begin{align}
\label{9_1_4a}
\liminf _{\sigma\rightarrow \infty}& J^{\eps}_2  \leq 3\eps/100,
\end{align} 
and \eqref{eq:bound-G-0_a1} follows.

To bound $\wG_0^{(3)}$,
once again by Girsanov's theorem and recalling the definition of $J^{\eps}_{2,2}$, we 
obtain
\begin{align*}
\nonumber
\Pm_v&(\wG^{(3)}_0)=\E_w\Big[\exp
\left(Z_{\sigma^8(T_\eps^\ast\wedge \xi^{\eps})}\right)
\1_{\wG^{(3)}_0}\Big]
\leq \E_w\left[\exp\left(\sigma^{-4}(M^f_{\sigma^8(T_\eps^\ast\wedge \xi^{\eps})}
-\frac{1}{2}\sigma^{-4} 
  A^f_{\sigma^8(T_\eps^\ast\wedge \xi^{\eps})})\right) \right. \\
&\hspace{1.5cm} \times\1\Big(\sup_{0\le t\leq T_\eps^\ast} 
M^{f,\sigma^4}\geq N_\eps \Big)\Big]   
= J^{\eps}_{2,2}
\leq \eps/50,
\end{align*}
where the last inequality follows from~\eqref{9_1_7} for $N_\eps$ sufficiently 
large and all $\sigma$ sufficiently large. 
Thus  \eqref{eq:bound-G-0_a3} follows, and the proof of Lemma~\ref{lem:13_1} is complete.
$\qed$


\def\cprime{$'$} \def\cprime{$'$} \def\cprime{$'$}
\providecommand{\bysame}{\leavevmode\hbox to3em{\hrulefill}\thinspace}
\providecommand{\MR}{\relax\ifhmode\unskip\space\fi MR }
\providecommand{\MRhref}[2]{%
  \href{http://www.ams.org/mathscinet-getitem?mr=#1}{#2}
}
\providecommand{\href}[2]{#2}

\end{document}